\documentclass[11pt]{article}
\RequirePackage{amsmath} %
\usepackage[affil-it]{authblk}
\usepackage{amsmath,amssymb,latexsym,amsfonts,color,enumitem}
\usepackage{hyperref,lmodern,graphicx,float,geometry,amsthm}
\usepackage{tikz,pgfplots,subcaption}
\newcommand{\ip}[2]{\langle #1 , #2 \rangle}    % inner product

\newcommand{\st}{\operatorname{s.t.}}

\newcommand{\cT}{{\mathcal{T}}}
\newcommand{\cN}{{\mathcal{N}}}

\newcommand{\cK}{{\mathcal{K}}}
\newcommand{\cP}{{\mathcal{P}}}

\newcommand{\cC}{{\mathcal C}}

\newcommand{\Sym}{\mathbb S}
\newcommand{\Y}{\mathcal Y}
\newcommand\N{{\mathbb N}}
\newcommand\R{{\mathbb R}}

\newcommand{\tr}{^{\top}}

\newcommand{\eps}{\varepsilon}
\newtheorem{theorem}{Theorem}
\newtheorem{lemma}{Lemma}
\newtheorem{proposition}{Proposition}
\newtheorem{definition}{Definition}
\newtheorem{remark}{Remark}

\newcommand{\copos}{\operatorname{\mathcal {C}\!{\it O}\!\cP}}

\newcommand{\interior}{\operatorname{int}}

\newcommand{\psd}[1]{\operatorname{\mathcal {P}\!{\it S}\!\mathcal{D}}\left(#1\right)}
\newcommand{\stack}{\operatorname{\textbf{Stk}}}
\newcommand{\slice}{\operatorname{\textbf{Slc}}}
\newcommand{\lrp}[1]{\left ( {#1} \right )}

\newcommand{\lrc}[1]{\left \{ {#1} \right \}}

\definecolor{mygreen}{cmyk}{0.82,0.11,1,0.25}%

\title{Positive semidefinite approximations to the cone of copositive kernels.}

\author{  Olga Kuryatnikova  \thanks{ \href{mailto:o.kuryatnikova@uvt.nl}{o.kuryatnikova@uvt.nl} (corresponding author),
        Tel.:+31134662430,
ORCID:  0000-0001-8460-7296} \quad and Juan C. Vera \thanks{\href{mailto:j.c.veralizcano@uvt.nl}{j.c.veralizcano@uvt.nl}} 
 \and 
\small{ Department of Econometrics and Operations Research, Tilburg University, 5037 AB Tilburg, Netherlands}
}

\begin{document}
\maketitle
%\tableofcontents
\begin{abstract}
It has been shown that the maximum stable set problem in some infinite graphs, and the kissing number problem in particular, reduces to a minimization problem over the cone of copositive kernels. Optimizing over this infinite dimensional cone is not tractable, and approximations of this cone have been hardly considered in literature. We propose two convergent hierarchies of subsets of copositive kernels, in terms of non-negative and positive definite kernels. We use these hierarchies and representation theorems for invariant positive definite kernels on the sphere to construct new SDP-based bounds on the kissing number.  This results in fast-to-compute upper bounds on the kissing number that lie between the currently existing LP and SDP bounds.
\end{abstract} \vspace{0.3cm}

 \vspace{0.5cm}

\section{Introduction} \label{sec:intro}
Copositive optimization is the class of linear optimization problems over the cone of {\it copositive} matrices or the dual of this cone. A matrix is called copositive if its quadratic form is non-negative over the non-negative orthant.
Copositive optimization is a powerful model, it captures several other subfields of optimization, such as quadratic, discrete and stochastic optimization see \cite{BomzeSurvey,DuerSurvey}.

Many NP-hard discrete optimization problems can be represented as copositive optimization problems (see, e.g.,  \cite{Bai16,Burer09,Burer12,KP,DistAvoid,QuadAsRef,ChromRef,Pena15,FracChromRef,Xu18}), which allows using tools from convex continuous optimization to tackle those discrete problems.
For example, a variety of graph parameters, such as the independence number \cite{KP}, the chromatic number \cite{ChromRef} and the fractional chromatic number \cite{FracChromRef}, have copositive formulations. Such formulations have also been constructed for the quadratic assignment problem \cite{QuadAsRef} and some distributional robust optimization problems \cite{Xu18}. In general, several large classes of integer quadratic problems can be written as copositive programs \cite{Bai16,Burer09,Burer12,Pena15}.

We are interested in solution methods for infinite dimensional copositive optimization, that is the optimization model obtained by replacing (finitely dimensional) copositive matrices with copositive \emph{kernels}, which are their infinite dimensional counterpart. %For infinite graphs, some of the graph parameters can be bounded using positive definite kernels. Moreover, several such problems are already known to have infinite dimensional copositive formulations \cite{ADDcite}.
Generalizing copositive optimization to infinite dimensions is inspired by some successful infinite dimensional generalizations of \emph{positive-semidefinite  optimization} (PSO). Such generalizations have proven useful in obtaining bounds for graph parameters in infinite graphs, by formulating an infinite dimensional version of well known PSO relaxations.
%. The approach is to replace PSD matrices by \emph{positive semi-definite (PSD) kernels}, %(also known as positive definite and/or positive kernels), which form a subset of the set of copositive kernels.
One of the applications of PSD kernels is generalizing the famous Lov{\'a}sz $\vartheta$-number~\cite{Lovasz} from finite graphs to certain types of infinite graphs. This fact has motivated some of the new results in packing problems in discrete geometry \cite{Packing_thesis}, the bounds on the measurable chromatic number~\cite{LbChr} and the measurable stable set of infinite graphs~\cite{Positive_thesis}. In the finite case, some graph parameters for which Lov{\'a}sz $\vartheta$-number provides a bound, such as the stable set or the chromatic number, can be formulated using copositive optimization. %It turns out that some infinite graph theoretical problems can be formulated as infinite dimensional copositive programs.
In the infinite case, the stable set problem in topological packing graphs \cite{alphaInf} and the measurable stable set problem in locally-independent graphs \cite{DistAvoid} have been formulated using infinite dimensional copositive optimization.
%{\color{red} Also, the relaxation of the spherical codes problem (and the kissing number problem, in particular) used by Pfender~\cite{Pfender}, could be seen as a relaxation of a admits both PSD and copositive kernels as feasible solutions.}
%\juan{I am not sure about the point here!}
 We expect that in future more problems will be represented using infinite dimensional copositive optimization and thus our results will be useful there.

Several methods have been proposed to approximately solve finite dimensional copositive optimization. The most usual approach is to approximate the copositive cone from the inside \cite{KP,Parrilo,alphaPVZ} or from the outside \cite{Outer_Las,Outer_Y}. Some researchers also exploit the structure of the problem and properties of the objective function \cite{SimplPart,coposKKT}. In the infinitely dimensional case,
there are not many approximations for the cone of copositive kernels.
The only known approach is to replace this cone by the better studied  cone of PSD kernels which is a subset of the cone of copositive kernels. When the kernels are defined on the unit sphere in $\R^n$, this results in a tractable relaxation of an infinite dimensional copositive program using the characterization of PSD kernels by Schoenberg~\cite{Schoenberg}. %For the stable set problem, this approach provides the already mentioned Lov{\'a}sz $\vartheta$-number~\cite{Lovasz} and its extensions.

%In this paper two hierarchies approximating the cone of copositive kernels (see Section~\ref{sec:qc}) are introduced. These hierarchies generalize existing hierarchies for finite dimensional case~\cite{toADD}.
%We use these hierarchies to construct upper bounds for the stable set problem by Dobre et al.~\cite{alphaInf}.

% We prove convergence of the resulting bounds (see Section~\ref{subsec:approx_sn}) and implement these bounds for the particular case of this problem (see Section~\ref{sec:kn}), namely, for the kissing number problem. Further, we provide more information on the main contributions of this paper.

One of the main contributions, In Section~\ref{sec:qc}, is the definition of two converging \emph{inner} hierarchies of \emph{subsets} of the cone of copositive kernels on $V{\times}V$ for any compact $V \subset \R^n$ to approximate the cone of copositive kernels.
Our inner approximations generalize two existing inner hierarchies for copositive matrices by De Klerk and Pasechnik~\cite{KP} and Pe{\~n}a et al. \cite{alphaPVZ}.%\juan{citation format??}.
 The key element of  our approach is to redefine the approximations using tensors. %More details about the approach is provided in Section~\ref{sec:prelim}.
We also show that the new hierarchies provide converging upper bounds for the stable set problem when applied to the results by Dobre et al.~\cite{alphaInf}.

Another important contribution, is the application of the proposed hierarchies of copositive kernels to construct converging upper bounds on the kissing number and provide numerical results for the first two levels of the hierarchy (see Section~\ref{sec:kn}).
%We derive an explicit expression which describes a subset of $Q^{S^{n-1}}_r$ using invariance of the kissing number problem under the action of the symmetry group of $S^{n-1}$ and Theorem~\ref{thm:PSD_multiv}.
The bound obtained using the level zero of the hierarchy  coincides with Delsarte, Goethals and Seidel \cite{DGS} linear programming (LP) bound for the spherical codes problem. The bound for level 1 is similar to the semidefinite programming (SDP) bound for the spherical codes problem by Bachoc and Vallentin \cite{BV}, but is somewhat weaker (see Section~\ref{sec:connections}).
 %This bound does not coincide with any of the existing bounds, but we can only implement a restricted version of the corresponding problem, which, unfortunately, does not provide any new bounds on the kissing number.
% The formulations of the bounds are not exactly comparable as they come from different formulations of the underlying problem. We also implement$Q^{S^{n-1}}_r$ for $r\ge 7$, but could not improve any of the existing upper bounds on the kissing number. \par
De Laat in his PhD thesis \cite{Packing_thesis} and in the related papers with Vallentin~\cite{Packing} and De Oliveira Filho~\cite{Packing,Packing2} provides a different type of PSD kernel based approximation for the stable set problem of compact topological packing graphs which  are not explicitly based on approximating the cone of copositive kernels.
The relation between these different approximations is an interesting question; some approximation results for the kissing number problem, which is a particular case of the stable set problem on infinite graphs,  are discussed in Section~\ref{sec:kn}.

The outline of the paper is as follows. In Section~\ref{sec:prelim} we introduce the notation and provide more detail on copositive and positive semi-definite kernels, as well as on tensors and tensor operators. In Section~\ref{sec:qc} we introduce generalized hierarchies  \eqref{def:CrInf} and \eqref{def:QrInf}, describe their main properties in Theorem~\ref{thm:convergence1} and show that they provide converging upper bounds for the stable set problem by Dobre et al. \cite{alphaInf} in Theorem~\ref{thm:convergence2}. In Section~\ref{sec:kn}, we describe implementation of the hierarchies for the kissing number problem and show computational results. In Section~\ref{sec:connections} we compare our optimization problems with some well known existing upper bounds on the kissing number. Section~\ref{sec:proofs} contains some technical proofs omitted throughout the paper. All computations in this paper are done in MATLAB R2017a on a computer with the processor  Intel\textsuperscript{\textregistered} Core\textsuperscript{\tiny{TM}} i5-3210M CPU @ 2.5 GHz and 7.7 GiB of RAM. SDP programs are solved with MOSEK Version 8.0.0.64. All MATLAB codes would be provided by request to the first author. \vspace{0.5cm}
%%%%%%%%%%%%%%%%%%%%%%%%%%%%%%%%%%%%%%%%%%%%%%%%%%%%%%%%%%%%%%%%%%%%%%%%%%%%%%%%%%%%%%%%%%%%%%%%
%%%%%%%%%%%%%%%%%%%%%%%%%%%%%%%%%%%%%%%%%%%%%%%%%%%%%%%%%%%%%%%%%%%%%%%%%%%%%%%%%%%%%%%%%%%%%%%%%

\section{Preliminaries} \label{sec:prelim}

We use $\R$ and $\R_{+}$ to denote the sets of real and nonnegative real numbers respectively. We also use $[n]$ to denote the set $\{1,\dots,n\}$.
 %and the sets of natural and positive natural numbers by $\N$ and $\N_{+}$ respectively.
 Let $\Sym^n$ be the space of $n \times n$ symmetric matrices over $\R$.
A matrix $M \in \Sym^n$ is called \emph{positive semidefinite (PSD)}  if for all $x \in \mathbb{R}^n$,
$ x\tr Mx \ge 0$.
We use the notation $M \succeq 0$ if $M$ is PSD.
An $n{\times}n$ matrix $M$ is called copositive if $x\tr Mx \ge 0$ for all $x\in \R^n_+$.
One can immediately see that the set of PSD matrices and the set of copositive matrices are both convex cones. \par
Let  $r \ge 0$. The following subsets of the cone of copositive matrices were introduced by De Klerk and Pasechnik~\cite{KP} and Pe{\~n}a et al. \cite{alphaPVZ}  respectively:

\begin{align}
%\cK^n_r = & \big\{M \in \Sym^n: \bigg ( \sum_{i=1}^n x_i^2\bigg )^r \sum_{i=1}^n \sum_{j=1}^n M_{ij}x_i^2  x_j^2  \text{ is a sum of squares} \big\}. \label{eq:defKr}\\
\cC^n_r =& \big\{M \in \Sym^n: (e\tr x)^r (x\tr M x)  \text{ has nonnegative coefficients} \big \},   \label{def:Cr} \\
 Q^n_r = & \big\{M \in \Sym^n: (e\tr x)^r (x\tr M x) =  \sum_{|\beta|=r} x^\beta x\tr N_\beta x + \sum_{|\beta|=r} x^\beta x\tr S_\beta x, \label{def:Qr} \\
& \qquad\qquad\quad N_\beta, S_\beta \in \Sym^n, \ N_\beta \ge 0 \text{ and } S_\beta \succeq 0 \text{ for all }\beta \in \mathbb{N}^n, |\beta|=r \big\}, \nonumber
 \end{align}
where $e$ is the vector of all ones, $|\beta| := \beta_1+\dots + \beta_n$ and $x^\beta := x_1^{\beta_1}\dots x_n^{\beta_n}$. From the definitions one can see that $\cC^n_r  \subseteq  Q^n_r$, $\cC^n_r  \subseteq  \cC^n_{r+1}$ and $Q^n_r  \subseteq  Q^n_{r+1}$. Moreover, the interior of the copositive cone is captured by both hierarchies of sets, that is the interior of the copositive cone is contained in $\bigcup_{r}\cC^n_r \subset \bigcup_{r}Q^n_r$  \cite{KP,alphaPVZ}.\vspace{0.5cm}

%%%%%%%%%%%%%%%%%%%%%%%%%%%%%%%%%%%%%%%%%%%%%%%%%%%%%%%%%%%%%%%%%%%%%%%%%%%%%%%%%%%%%%%%%%%%%%%%
\subsection{Positive definite and copositive kernels} \label{subsec: prelim_psdCop}

We are interested in the infinite dimensional version of copositivity. Let $V \subset \R^n$ be a compact set.
% Let $V \subset \R^n$ be a compact set endowed with a finite measure $\mu$ strictly positive on open sets.
 Denote the set of real-valued continuous functions on $V$ by $C(V)$. We use the name \emph{kernels on $V$} for the set of symmetry real continuous functions on $V{\times}V$:
\[\cK(V)=\{F \in C(V{\times}V) :  F(x,y)=F(y,x), \ \forall \ x,y \in V\}.\]
For any finite $U$ of size $n$,  $\cK(U)$ is isomorphic to the set of  symmetric  $n \times n$ matrices; we will abuse the notation modulo this isomorphism, and thus we do no distinguish between kernel over finite sets and matrices.   Given $K \in \cK(V)$ and $U \subseteq V$ we denote by $K^U$ the restriction of $K$ to $U\times U$. Notice that for all $U \subseteq V$ and all $K \in \cK(V)$ we have that $K^U \in \cK(U)$.
%A kernel is positive semi-definite if  the matrix induced by restricting this kernel to a finite subset $U\times U, \ U\subseteq V$ is positive semi-definite for any finite $ U\subseteq V$.
%A natural question is how to to extend the notion of copositivity to kernels on sets other than $[n]$. The answer was given in \cite{alphaInf}: a kernel $K$ is \emph{copositive} if the matrix induced by restricting $K$ to a finite subset $U\times U, \ U\subseteq V$ is copositive for any finite $ U\subseteq V$. It is again clear that copositive kernels form a convex cone, and we denote this cone by $\copos(V)$. For a finite $V$, $\copos(V)$ is isomorphic to the cone of copositive matrices of size $|V|$. By using kernels in $(CP)$, we obtain an infinite dimensional copositive optimization.
%Further in the paper the notation $\mu$ always denotes this measure.%, unless otherwise stated. % Let $C(V): V \rightarrow \R$ denote the set of real-valued continuous functions on $V$, and $C_+(V)$ denote the set of nonnegative continuous functions.
% Following \cite{alphaInf}, we  introduce the set of \emph{kernels}:
%$$\cK(V)=\{K \in C(V{\times}V) :  K(x,y)=K(y,x), \ \forall \ x,y \in V\}.$$
%Kernels can be seen as an extension of the real symmetric matrices $\Sym^n$: $\cK([n]) \cong\Sym^n$. \par

A kernel $K$ on $V$ is  \emph{positive semi-definite (PSD)} if for any finite $U \subset V$ the matrix $K^U$ is PSD. In literature it is rather common to use the notation \emph{positive definite} \cite{FinPSD,MultPsd,Schoenberg}, however, we deviate from this convention in order to have correspondence in the notation for matrices and kernels. We denote the cone of PSD kernels on $V$ by $\psd{V}$.

Now we turn our attention to copositivity. Copositive kernels were introduced by Dobre et al.~\cite{alphaInf}, a kernel is copositive if and only if for any finite $U \subset V$ the matrix $K^U$ is copositive. We denote the set of copositive kernels on $V$ by $\copos(V)$.  Notice that $\psd{V} \subset \copos(V)$.\vspace{0.5cm}

\subsection{Tensor operators and their properties} \label{subsec:prelim_tens}
We use tensor notation and terminology similar to Dong \cite{TensorQ}. For a positive number $r \in \mathbb{N}$, a {\em tensor of order $r \in \mathbb{N}$ over a set $V$} is a real valued function on $V^r$. Denote by $\cT^V_r$ the set of r-tensors over the set $V$. For $r = 2$ and $N=[n]$, a tensor is just an
 $n\times n$ matrix, i.e., $\cT^{[n]}_2 \cong \R^{n \times n}$. Thus, for readability, similar to the case of kernel over finite sets,  we abuse the notation and consider the notions of 2-tensors and matrices interchangeable.

Next, we introduce two tensor operators.  The first operator is a lifting operator.
For $r \ge 0$, we define the \emph{$r$-stack}, $\stack^\mathbf{r}: \cT^V_{d} \rightarrow \cT^V_{d+r}$ by
\begin{align}
\stack^\mathbf{r}(T)(u,v):=T(u), \ \ \text{for all } T \in  \cT^V_{d},\, u \in V^d, \,  v \in V^r. \label{def:stack}
\end{align}
Notice that $\stack^0(T):=T$. The second operator is the symmetrization operator. Denote by $Sym(d)$ the group of permutations on  $d$ elements. %For $v \in V^d$, define $\pi v=v_{\pi(1),\dots,\pi(d)}$.
We define the \emph{symmetrization} $\sigma: \cT^V_d \to \cT^V_d$ by
\begin{align}
\sigma(T)(v): = \frac 1{d!} \sum_{\pi \in Sym(d)} T(\pi v ). \label{def:Sigma}
\end{align}
%These operators have the following properties which we use throughout the paper:
\begin{lemma} \label{lem:SigmaStack}  Let $d >0 $, $r \ge 0$, $V \subset \R^n$, $T,S \in \cT_d^V$. Then
\begin{enumerate}[label=\alph*.]
\item $\sigma(T+S)=\sigma(T)+\sigma(S).$
\item $\stack^\mathbf{r}(T+S)=\stack^\mathbf{r}(T)+\stack^\mathbf{r}(S).$
%\item $\textbf{I}^{V,\mu}\lrp{T+S,f})=\textbf{I}^{V,\mu}\lrp{T,f}+\textbf{I}^{V,\mu}\lrp{S,f}$
\item $\stack^\mathbf{r+1}(T)=\stack^\mathbf{1}\lrp{\stack^\mathbf{r}(T)}.$
\item If $\sigma(T)=\sigma(S)$, then $\sigma \lrp{ \stack^\mathbf{r}(T)}=\sigma \lrp{ \stack^\mathbf{r}(S)}$.
%\item $\textbf{I}^{V,\mu}\lrp{\sigma(T),f})=\textbf{I}^{V,\mu}\lrp{T,f}$
%\item $\textbf{I}^{V,\mu}\lrp{T,f}\textbf{I}^{V,\mu}\lrp{1,f}=\textbf{I}^{V,\mu}\lrp{\stack^\mathbf{r}(T),f}$
\end{enumerate}
\end{lemma}
\begin{proof}
$a.$, $b.$ and $c.$ are straightforward.  To prove $d.$, assume $\sigma(T)=\sigma(S)$. For $r=0$,
\[\sigma \lrp{ \stack^\mathbf{r}(T)}=\sigma(T)=\sigma(S)=\sigma \lrp{ \stack^\mathbf{r}(S)}.\]
Using $c.$ and induction, it is enough now to prove the statement for $r=1$. For any $v \in V^{d+1}$,
\begin{align*}
\sigma\left (\stack^\mathbf{1}(T)\right)(v) =& \frac{1}{(d+1)!} \sum_{\pi \in Sym(d+1)}  \stack^\mathbf{1}(T) (\pi v )\\
%=& \frac{1}{(d+1)!} \sum_{\pi \in Sym(d+1)}  \stack^\mathbf{1}(T) (v_{\pi(1)},\dots,v_{\pi (d+1)} )\\
%=& \frac{1}{(d+1)!} \sum_{k=1}^{d+1}  \sum_{\pi \in Sym(d+1), \ \pi(d+1)=k}  \stack^\mathbf{1}(T) (v_{\pi(1)},\dots,v_{\pi (d+1)} )\\
=& \frac{1}{(d+1)!} \sum_{k=1}^{d+1}  \sum_{\pi \in Sym(d+1), \ \pi(d+1)=k}  T(v_{\pi(1)},\dots,v_{\pi (d)} )\\
%=& \frac 1{(d+1)!} \bigg( \sum_{k=1}^{d+1}  \sum_{\pi \in Sym(d)}  T \big (\pi (v_1,\dots,v_{k-1},v_{k+1},\dots,v_{d+1}) \big )\bigg)\\
=&\frac{d!}{(d+1)!} \bigg( \sum_{k=1}^{d+1} \sigma(T)  (v_1,\dots,v_{k-1},v_{k+1},\dots,v_{d+1})\bigg)\\
=&\frac{d!}{(d+1)!} \bigg( \sum_{k=1}^{d+1} \sigma(S)  (v_1,\dots,v_{k-1},v_{k+1},\dots,v_{d+1})\bigg) \\
=& \sigma\left (\stack^\mathbf{1}(S)\right)(v).
\end{align*}
%Using $c.$, $\sigma \lrp{ \stack^\mathbf{r}(T)}=\sigma \lrp{ \stack^\mathbf{r}(S)}$  by induction.
%\qed
\end{proof}

%%%%%%%%%%%%%%%%%%%%%%%%%%%%%%%%%%%%%%%%%%%%%%%%%%%%%%%%%%%%%%%%%%%%%%%%%%%%%%%%%%%%%%%%%%%%%%%%
%%%%%%%%%%%%%%%%%%%%%%%%%%%%%%%%%%%%%%%%%%%%%%%%%%%%%%%%%%%%%%%%%%%%%%%%%%%%%%%%%%%%%%%%%%%%%%%%%

\section{Inner hierarchies for the cone of copositive kernels} \label{sec:qc}

To define our hierarchies approximating the copositive cone,
we use the previously defined operators together with the following two sets of tensors.
\begin{definition} \label{def:nonneg}
A tensor $T \in  \cT^V_{r}$ is \emph{non-negative} if $T(v) \ge 0, \ \forall \ v \in V^r$. We use $\cN^V_r$ to denote the set of all non-negative r-tensors.
\end{definition}

\begin{definition} \label{def:2psd}
A  tensor $F \in  \cT^V_{r+2}$ is {\em 2-PSD on $V$} if it is continuous, and for all $z\in V^r$, $F(\cdot,\cdot,z) \in \psd{V}$.
\end{definition}

\noindent Let us introduce the following sets:
\begin{align}
\cC^V_r = & \big \{M \in \cK(V):  \sigma \lrp{\stack^\mathbf{r}(M)}= \sigma (N), \ N \in \cN^{V}_{r+2}\big \}, \label{def:CrInf} \\
Q^V_r =  \label{def:QrInf} & \big \{M \in \cK(V):  \sigma \lrp{\stack^\mathbf{r}(M)}= \sigma (N) + \sigma(S), \\
\nonumber&\qquad\qquad\qquad\qquad\ N \in \cN^{V}_{r+2} \text{ and } S\in \cT^{V}_{r+2}  \text{ is 2-PSD } \big \}.
\end{align}
Proposition \ref{prop:QrTensor} shows that our constructions generalize \cite{alphaPVZ,KP}.  For the case of matrices (i.e. using finite tensors) Dong \cite{TensorQ} pursued similar ideas, but restricted to the set of tensors invariant under variable permutations.

\begin{proposition}\label{prop:QrTensor} For any $r \in \mathbb{N}$,
\begin{align*}\cC_r^n = \cC^{[n]}_r, \ Q_r^n = Q^{[n]}_r.
\end{align*}
\end{proposition}
The proof of Proposition~\ref{prop:QrTensor}
is provided in Section~\ref{sec:proofs}. % \\

For a set $S\subseteq \R^n$, denote by $\interior{S}$ the interior of $S$.  Next theorem shows that properties of $\cC^{[n]}_r,Q^{[n]}_r$ proven in \cite{KP} and \cite{alphaPVZ} respectively can be generalized for a general compact $V$.

\begin{theorem}\label{thm:convergence1} Let $V\subset \R^n$ be a compact set. Then,
\[\cC^V_0 \subseteq \cC^V_1 \subseteq \dots \subseteq\copos(V),  \ \ Q^V_0 \subseteq Q^V_1 \subseteq \dots \subseteq\copos(V)\]
\[ \text{ and }  \interior{\copos(V)} \subseteq \bigcup_r  \cC^V_r \subseteq  \bigcup_r  Q^V_r. \label{inclusion2}  \]
\end{theorem}

The key ingredient in the proof of Theorem \ref{thm:convergence1} is the characterization of the interior of the copositive cone given in Proposition \ref{prop:interior}.
When $V$ is finite (i.e. for matrices), the interior of the copositive cone consists of those copositive matrices whose quadratic form is strictly positive on the standard simplex  $ \Delta^n:=\{x \in \R^n:  \ e \tr x=1, \ x \ge 0\}$. This implies, by compactness of the simplex, that $M \in \interior{\copos(V)}$ if and only if there is $\epsilon >0$ such that  $x^TMx = \sum_{v \in V}\sum_{u \in V} M(v,u)x_ix_j \ge \epsilon$ for all $x \in \Delta^n$. Proposition \ref{prop:interior} shows that in the general case, for compact $V$, this is true uniformly over all finite submatrices of the kernel $M$.

\begin{proposition} \label{prop:interior}
Let $V\in R^n$ be a compact set. Then
\begin{align}
\interior{\copos(V)}{=}&\big\{ M\in \cK(V): \text{there is } \eps>0 \text{ such that for all } n>0  \hspace{2cm} \nonumber\\
&\hspace{0.2cm} \text{and all } v_1,\dots,v_n\hspace{0.05cm}{\in}\hspace{0.05cm}V, \ x \hspace{0.05cm}{\in}\hspace{0.05cm}\Delta^n\hspace{-0.1cm}: \sum_{i=1}^n\sum_{j=1}^n M(v_i,v_j)x_ix_j \ge \epsilon  \big\} \label{def:int}
 \end{align}
\end{proposition}
 \begin{proof}
Let $M, \eps$ be such that \eqref{def:int} holds. Let $K\in \cK(V)$ be given. Since $K$ is continuous and $V^2$ is compact, $K$ attains its maximum on $V{\times}V$. Let $K^*=\max_{x,y \in V} K(x,y)$. Then for any $v_1,\dots,v_n \in V$ and $x\in \Delta^n$,
\begin{align*}
&\sum_{i=1}^n\sum_{j=1}^n \lrp{M(v_i,v_j)-\frac{\eps}{K^*}K(v_i,v_j)}x_ix_j\\
&\qquad\ge  \eps-\frac{\eps}{K^*} \lrp{\max_{i,j\in [n]}K(v_i,v_j) \sum_{i=1}^n\sum_{j=1}^n x_ix_j}\\
&\qquad\ge 0.
 \end{align*}
Hence $M- \frac{\eps}{K^*}K  \in \copos(V)$ by definition,  and thus $M\in \interior{\copos(V)}$.
Now let $M\in \interior{\copos(V)}$. Then there is $\eps>0$ such that $M-\eps \mathbb{J} \in \copos(V)$. Therefore for any choice of $v_1,\dots,v_n \in V$,
\begin{align*}
& \min_{x \in \Delta^n} \sum_{i=1}^n\sum_{j=1}^n M(v_i,v_j)x_ix_j\\
&\qquad= \min_{x \in \Delta^n} \sum_{i=1}^n\sum_{j=1}^n \lrp{M-\eps\mathbb{J}}(v_i,v_j)x_ix_j +\eps \sum_{i=1}^n\sum_{j=1}^n x_ix_j\\
&\qquad\ge \eps.
 \end{align*}
%\qed
\end{proof} \vspace{0.2cm}

To prove Theorem \ref{thm:convergence1}, we need some  additional results. A result by Powers and Reznick~\cite{Reznick} on the rate of  convergence in P{\'o}lya's theorem (see, e.g.,~\cite{HardLP88}), and a  characterization of $\cC_r^V$ in terms of $\cC_r^U$ for all finite $U \subset V$ (see the definitions of $\psd{V}$ and $\copos(V)$ given in Section~\ref{subsec: prelim_psdCop}).

\begin{lemma}\label{lem:Csubset} Let $V\subset \R^n$ be a compact set, and let $r \in \mathbb{N}$. $K\in \cC^V_r$ if an only if for any finite $U \subset V$, the matrix $K^U \in \cC^{U}_r$.
\end{lemma}
\begin{proof} Let $U \subset V$ be finite. If $K\in \cC^V_r$, then $\sigma\lrp{\stack^\mathbf{r}(K)}=\sigma(N)$ for some $N \in \cN^V_{r+2}$. Therefore $\sigma\lrp{\stack^\mathbf{r}(K^U)}=\sigma(N^U)$, that is $K^U \in \cC^{U}_r$.

Now, assume $K^U\in \cC^{U}_r$ for each finite $U \subset V$. Let $N = \sigma\lrp{\stack^\mathbf{r}(K)}$. Then $\sigma(N)=N$. For any
given $v_1,\dots,v_{r+2} \in V$, let $U = \{v_1,\dots,v_{r+2}\}$.
Then,
\begin{align*}
N(v_1\dots,v_{r+2})= & \ \sigma\lrp{\stack^\mathbf{r}(K)}(v_1\dots,v_{r+2})\\
= & \ \sigma\lrp{\stack^\mathbf{r}(K^U)}(v_1\dots,v_{r+2})\ge 0.
\end{align*}
Hence $K\in \cC^V_r$.
%\qed
\end{proof}

\begin{proposition}[Powers and Reznick~\cite{Reznick}] \label{prop:Polya_bound} Let $M \in \Sym^n$ be  strictly copositive. Then the polynomial
$(e\tr x)^r\sum_{i,j=1}^{|U|}M_{ij}x_ix_j$
has only positive coefficients if $r> \frac{L}{k}-2,$
where
$L=\max_{ij}|M_{ij}|$
and
$k=\min_{x \in \Delta^n}x \tr M x.$
\end{proposition}

\begin{remark} The bound on $r$ in Proposition \ref{prop:Polya_bound} does not depend on the size of $M$.
\end{remark}

\noindent Now we are ready to prove Theorem~\ref{thm:convergence1}.

\begingroup
\renewcommand*{\proofname}{\textbf{Proof of Theorem \ref{thm:convergence1}}}
\begin{proof} Let $r \ge 0$ and let  $M \in Q^V_r$. Then by~\eqref{def:QrInf} there exists $N \in \cN^V_{r+2}, \ S\in \cT^V_{r+2}$ such that $S$ is 2-PSD and $\sigma \lrp{\stack^\mathbf{r}(M)} =
\sigma \lrp{N}+\sigma \lrp{S}$.
First, we show that $M \in Q^V_{r+1}$. By  $c.$ and $d.$ in Lemma \ref{lem:SigmaStack},
\[
\sigma \lrp{\stack^\mathbf{r+1}(M)} =
\sigma \lrp{\stack^{1}(N)}+\sigma \lrp{\stack^{1}(S)}.
\]
From the definition \eqref{def:stack} of the stack operator, $\stack^{1}(S)\in \cT^V_{r+3}$ is 2-PSD and $\stack^{1}(N) \in \cN^V_{r+3}$. Thus $M \in Q^V_{r+1}$. Analogously, $\cC^V_r \subseteq \cC^V_{r+1}$.

\medskip
\noindent Now, we show $M \in \copos(V)$.  Let $v_1,\dots,v_n \in V$ and $x\in \R^n_+$, then
\begin{align*}
\hspace{3em}& \hspace{-3em}(e\tr x)^r \hspace{-0.2cm} \sum_{i,j\in [n]}M(v_i,v_j)x_ix_j\\
&= \sum_{k \in [n]^{r+2}} x_{k_1}\cdots x_{k_{r+2}} M(v_{k_1},v_{k_2}) \\
&=   \sum_{k \in [n]^{r+2}} x_{k_1}\cdots x_{k_{r+2}}  \stack^\mathbf{r}(M)(v_{k_1},\dots, v_{k_{r+2}}) \\
&= \hspace{-0.05cm} \tfrac{1}{(r+2)!}\hspace{-0.25cm}\sum_{\pi \in Sym(r+2)}  \sum_{k \in [n]^{r+2}}\hspace{-0.2cm} x_{k_{\pi(1)}}\cdots x_{k_{\pi(r+2)}}  \stack^\mathbf{r}(M)(v_{k_{\pi(1)}},\dots, v_{k_{\pi(r+2)}}) \\
&=   \sum_{k \in [n]^{r+2}} x_{k_1}\cdots x_{k_{r+2}}  \sigma\lrp{\stack^\mathbf{r}(M)}(v_{k_1},\dots, v_{k_{r+2}}) %\\
%&=  \sum_{k \in [n]^{r+2}}x_{k_1}\dots x_{k_{r+2}}  \sigma\lrp{S+N}(v_{k_1},\dots, v_{k_{r+2}})  \hspace{3em}
\end{align*}
Using $\sigma\lrp{\stack^\mathbf{r}(M)} = \sigma\lrp{S+N} \ge \sigma(S)$, from Definition~\ref{def:nonneg}, we obtain
\begin{align*}
\hspace{3em}& \hspace{-3em}(e\tr x)^r  \sum_{i,j\in [n]}M(v_i,v_j)x_ix_j\\
&\ge \sum_{k \in [n]^{r+2}}x_{k_1}\dots x_{k_{r+2}} \sigma(S)(v_{k_1},\dots, v_{k_{r+2}}) \\
&= \tfrac 1{(r+2)!} \sum_{k \in [n]^{r+2}}\sum_{\pi \in Sym(r+2)}  x_{k_1}\dots x_{k_{r+2}}  S(v_{ k_{\pi(1)}},\dots, v_{k_{\pi(r+2)}})  \\
&= \tfrac 1{(r+2)!} \sum_{k \in [n]^{r+2}}S(v_{k_{1}},\dots, v_{k_{r+2}}) \sum_{\pi \in Sym(r+2)}  x_{k_{\pi(1)}}\dots x_{k_{\pi(r+2)}} \\
&= \sum_{k \in [n]^{r+2}}S(v_{k_{1}},\dots, v_{k_{r+2}})  x_{k_{1}}\dots x_{k_{r+2}} \\
&= \sum_{k \in [n]^{r}}x_{k_{1}}\dots x_{k_{r}} \bigg(\sum_{i,j \in [n]}S(v_i,v_j,v_{k_1},\dots, v_{k_r})  x_ix_j\bigg)\\
&\ge 0,
\end{align*}
where the last inequality follows from Definition~\ref{def:2psd}).\\
We have then $M \in \copos(V)$. As $\cC^V_r \subseteq Q^V_r$ by construction of $ \cC^V_r$ \eqref{def:CrInf} and $Q^V_r$\eqref{def:QrInf}, we also have $\cC^V_r \subseteq \copos(V)$. \vspace{0.2cm}

\noindent For the final part of the proof, let $M\in \interior{\copos(V)}$.
Since $M$ is continuous and $V{\times}V$ is compact, $M$ attains its maximum and minimum values on $V{\times}V$.

\noindent Denote
\[L=\max_{x,y \in V}|M(x,y)|.\]
As $M\in \interior{\copos(V)}$, by Proposition~\ref{prop:interior} there is $\eps>0$ such that
\begin{align*}
\min_{x \in \Delta^{|U|}}x \tr M^U x \ge \eps,
\end{align*}
for all finite $U \subseteq V$.

Let $r > \tfrac{L}{\eps}-2$.
By Proposition \ref{prop:Polya_bound}, $M^U \in \cC^{U}_{r}$ for any $U \subseteq V$, which implies $M \in \cC^{V}_{r}$ by Lemma~\ref{lem:Csubset}. Thus,
$\interior\copos(V) \subseteq \cC^{V}_{r} \subseteq  Q^V_{r}$
%\qed
\end{proof}
\endgroup \vspace{0.5cm}

\subsection{Approximating the stability number of infinite graphs} \label{subsec:approx_sn}

The stability number of a graph is the maximum number of vertices such that no two of them are adjacent. In this subsection we introduce the formulation of the stability number problem on infinite graphs using copositive optimization considered in Dobre et al. \cite{alphaInf} and prove that if the problem is strictly feasible, then the stability number can be approximated as closely as desired by replacing $\copos(V)$ with $\cC_r^{V}$ or $Q_r^{V}$ with $r$ big enough.

Following de Laat and Vallentin \cite{Packing}, we define a \emph{compact topological packing graph} as the graph where the vertex set is a compact Hausdorff topological space, and every finite clique is contained in an open clique.  Stability number of these graphs is finite. One example of such graphs is the graph $\mathcal{G}^\theta_n=(S^{n-1},E^{\mathcal{G}^\theta_n})$ in which the vertex set is a unit sphere $S^{n-1}$ in $\R^n$ and $(u,v) \in E^{\mathcal{G}^\theta_n}$ if and only if $u \tr v \in (\cos{\theta},1)$. That is, there is an edge between every two vertices when the angle between them is strictly smaller than $\theta$. An example of a graph that is not a compact topological packing graph is the graph $\mathcal{H}^\theta_n=(S^{n-1},E^{\mathcal{H}^\theta_n})$ in which there is an edge between two vertices when the angle  between them is equal to $\theta$.  \par

{\theorem[Theorem 1.2. from Dobre et al. \cite{alphaInf}] Let $G=(V,E)$ be a compact topological packing graph. Then the stability number of $G$ equals
\begin{align}
\alpha(G) = \inf_{K \in  \cK(V) ,  \ \lambda \in \R}& \ \lambda \label{pr:Main}\\
\st \hspace{0.6cm}& \  K(v,v) = \lambda-1 &\text{for all }v \in V,\nonumber \\
&\ K(u,v) = -1 &\text{for all }(u,v) \not\in E, \nonumber \\
& \  K \in  \copos(V). \nonumber
\end{align}}
An example of the stability number problem on a compact topological graph  is  \emph{the spherical codes problem}. In the spherical codes problem, the number of points on the unit sphere in $\R^n$  for which the pairwise angular distance is not smaller than some value $\theta$ is maximized. \\
This problem can be viewed as the stable set problem on the graph $\mathcal{G}^\theta_n=(S^{n-1},E^{\mathcal{G}^\theta_n})$ introduced in the previous paragraph. A particular case of the spherical codes problem when $\theta=\frac{\pi}{3}$ is \emph{the kissing number} problem, which we analyze in detail later in Section~\ref{sec:kn}. \vspace{0.2cm}

Define the following relaxations to Problem~\eqref{pr:Main}:
\begin{align}
\gamma_r(G) = \inf& \ \lambda \label{pr:Gamma} \\
\st & \ K(v,v) = \lambda-1 &\text{for all }v \in V, \nonumber \\
&\ K(u,v) = -1 &\text{for all }(u,v) \not\in E, \nonumber\\
& \ K \in \cC^{V}_r. \nonumber
\end{align}
\begin{align}
\nu_r(G) = \inf& \ \lambda \label{pr:Nu} \\
\st & \ K(v,v) = \lambda-1 &\text{for all }v \in V, \nonumber \\
&\ K(u,v) = -1 &\text{for all }(u,v) \not\in E, \nonumber\\
& \ K \in Q^{V}_r. \nonumber
\end{align}
By Theorem \ref{thm:convergence1}, $ \cC^{V}_r \subseteq  Q^{V}_r \subseteq \copos(V)$ for any $r$. Hence,
\begin{align}
\alpha(G) \le \nu_r(G) \le \gamma_r(G). \label{eq:UpBound}
\end{align}
%Convergence of the upper bounds $\nu_r(G), \gamma_r(G)$ can be proven under certain conditions:
\begin{theorem} \label{thm:convergence2} Let $G=(V,E)$ be a compact topological packing graph and let Problem~\eqref{pr:Main} be strictly feasible.  That is, there exists feasible $\lrp{K^+,\lambda^+}$ such that $K^+\in \interior{\copos(V)}$. Then $\gamma_r(G) \downarrow \alpha(G)$ and  $\nu_r(G) \downarrow \alpha(G)$.
\end{theorem}
\begin{proof} From the definition of infimum, for any  $n > 0$ there is $\lrp{K_n,\lambda_n}$ feasible for \eqref{pr:Main}  such that $\lambda_n  \le \alpha(G)+\frac{1}{n}.$ Let
\[K^+_n =\tfrac{1}{n} K^+ + \lrp{1-\tfrac{1}{n}} K_n \quad \text{ and } \quad \lambda^+_n = \tfrac{1}{n} \lambda^+ + \lrp{1-\tfrac{1}{n}} \lambda_n.\]
By convexity of \eqref{pr:Main} the pair $(K^+_n,\lambda^+_n)$ is feasible for Problem~\eqref{pr:Main}.
Since $K^+ \in \interior{\copos(V)}$, by Theorem \ref{thm:convergence1} there is $m >0$ such that
$K^+_n  \in  \cC^V_{r}$ for all $r\ge m$. Hence,
\[\lim_{r \rightarrow \infty} \gamma_r(G) \le \lambda^+_n = \tfrac{1}{n}\lambda^+ + \lrp{1-\tfrac{1}{n}}\lambda_n\le \tfrac{1}{n}\lambda^+ + \lrp{1-\tfrac{1}{n}} \lrp{\alpha(G)+\tfrac{1}{n}} .\]
Taking the limit when $n\rightarrow \infty$ on both sides and using the second inequality in \eqref{eq:UpBound} we obtain
\[  \alpha(G)\ge \lim_{r \rightarrow \infty} \gamma_r(G) \ge \lim_{r \rightarrow \infty} \nu_r(G). \]
The first inequality in \eqref{eq:UpBound} concludes the proof.
%\qed
\end{proof} \vspace{0.5cm}

 %%%%%%%%%%%%%%%%%%%%%%%%%%%%%%%%%%%%%%%%%%%%%%%%%%%%%%%%%
 %%%%%%%%%%%%%%%%%%%%%%%%%%%%%%%%%%%%%%%%%%%%%%%%%%%%%%%%%
\section{Bounds on the kissing number problem} \label{sec:kn}
 The \emph{kissing number} is the maximum number $\kappa_n$
of non-overlapping unit spheres $S^{n-1}$ in $\R^n$ that can simultaneously touch other unit sphere. History of the problem is described in detail in, for instance, Musin \cite{KN4}. The value of $\kappa_n$ is known for $n = 1, 2, 3, 4, 8, 24$. Computing $\kappa_1=2$ and $\kappa_2=6$ is straightforward, but this is not the case for $\kappa_n$ with $n>2$. The question whether  $\kappa_3=12$ or $\kappa_3=13$ is attributed to the famous discussion between Isaac
Newton and David Gregory in 1694, and the result  $\kappa_3=12$ was proven only in 1953 by Sch{\"u}tte and van der Waerden \cite{KN3}. The numbers $\kappa_8=240$ and $ \kappa_{24} = 196560$ were found by Odlyzko and Sloane~\cite{KN824} and Levenshtein~\cite{KN824_2} in 1979. Finally, $\kappa_4=24$ was proven in 2003 by Musin \cite{KN4}. A lot of research has been done to approximate the kissing numbers in other dimensions from above and below. In this paper we are interested in obtaining upper bounds using the hierarchies proposed in Section~\ref{sec:qc}. \par
In 1977, Delsarte, Goethals and Seidel \cite{DGS} proposed an LP upper bound used later to obtain $\kappa_8$ and $\kappa_{24}$. To find this bound, one has to solve an infinite dimensional LP, which can be approximated by a semidefinite program (SDP). The LP bound is not tight in general, for example, for $\kappa_4$ it cannot be less than 25, as was shown by Arestov and Babenko~\cite{Dels_25} in 1997. Musin \cite{KN4} strengthened the LP bound for the four dimensional case to obtain $\kappa_4=24$. In 2007, Pfender~\cite{Pfender} strengthened the LP bound and obtained the best existing upper bounds for $n=25$ and $n=26$. In 2008, a new SDP upper bound was proposed by Bachoc and Vallentin~\cite{BV}. Mittelman and Vallentin \cite{MV} used this approach to compute upper bounds for $5\le n \le 23$. Later these results were strengthened for dimensions 9 to 23 by
Machado and de Oliveira Filho \cite{BestBound} who exploited symmetry of the SDP problem. Table~\ref{tab1} shows best known bounds on the kissing numbers in some dimensions, bounds for other dimensions can be found in \cite{MV} and \cite{BestBound}. Finally, Musin \cite{MultPsd} proposed a hierarchy generalizing the earlier mentioned linear and SDP approaches, but this hierarchy has not been implemented and is not proven to converge.
{\scriptsize{
\begin{table}[h]
\renewcommand{\tabcolsep}{2pt}
\centering
\caption{Best known upper and lower bounds on the kissing number.}
\label{tab1}
\begin{tabular}{|l|c|c|c|c|c|c|c|c|c|c|c|c|c|c|c|c|}
\hline
 \ $n$
%& 2
   & 3                                                         & 4                                                         & 5                                                        & 6                                                        & 7                                                         & 8                                                                      & 9                                                                & 10                                                               & 11                                                               & 12
& 13                                                                %& 14                                                                & 15
& 24
& 25
& 26                                                                        \\ \hline
 \begin{tabular}[c]{@{}c@{}}upper \\ \ bound \end{tabular}
%& 6
 & \begin{tabular}[c]{@{}c@{}}12\\ \cite{KN3}\end{tabular} & \begin{tabular}[c]{@{}c@{}}24\\ \cite{KN4}\end{tabular} & \begin{tabular}[c]{@{}c@{}}44\\ \cite{MV}\end{tabular} & \begin{tabular}[c]{@{}c@{}}78\\ \cite{BV}\end{tabular} & \begin{tabular}[c]{@{}c@{}}134\\ \cite{MV}\end{tabular} & \begin{tabular}[c]{@{}c@{}}240\\ \cite{KN824,KN824_2}\end{tabular} & \begin{tabular}[c]{@{}c@{}}363\\ \cite{BestBound}\end{tabular} & \begin{tabular}[c]{@{}c@{}}553\\ \cite{BestBound}\end{tabular} & \begin{tabular}[c]{@{}c@{}}869\\ \cite{BestBound}\end{tabular} & \begin{tabular}[c]{@{}c@{}}1356\\ \cite{BestBound}\end{tabular}
 & \begin{tabular}[c]{@{}c@{}}2066\\ \cite{BestBound}\end{tabular}
 % & \begin{tabular}[c]{@{}c@{}}3177\\ \cite{BestBound}\end{tabular} & \begin{tabular}[c]{@{}c@{}}4858\\ \cite{BestBound}\end{tabular}
  & \begin{tabular}[c]{@{}c@{}}196560\\ \cite{KN824,KN824_2}\end{tabular}
 & \begin{tabular}[c]{@{}c@{}}278083\\ \cite{Pfender}\end{tabular}
 & \begin{tabular}[c]{@{}c@{}}396447\\ \cite{Pfender}\end{tabular}
 \\ \hline
\begin{tabular}[c]{@{}c@{}}lower \\ \ bound \\ \cite{LB}\end{tabular}
 %& 6
 & 12                                                        & 24                                                        & 40                                                       & 72                                                       & 126                                                       & 240                                                                    & 306                                                              & 500                                                              & 582                                                              & 840
& 1154                                                              %& 1606                                                              & 2564
& 196560
  & 197040
  & 198480 \\ \hline
$\tfrac{ub{-}lb}{lb}$
%&0
 & 0                                                         & 0                                                         & 10                                                       & 0.08                                                      & 0.06                                                       & 0                                                                      & 0.19                                                               & 0.11                                                             & 0.49                                                             & 0.61
 & 0.79                                                                %& 97.8                                                              & 89.5
 & 0
 & 0.41
 & 0.997 \\ \hline
\end{tabular}
\end{table}
}}

The kissing number $\kappa_n$  can be reformulated as the independence number on a graph whose vertex set is the unit sphere: \vspace{0.05cm}

\begin{definition} \label{def:Gn} $\mathcal{G}_n=(S^{n-1},E)$ is the graph with the edge set
\[E{=}\lrc{(u,v)\in S^{n-1}\times S^{n-1}: u\tr v > \tfrac{1}{2}}.\]
\end{definition}
We have then $\alpha(\mathcal{G}_n) = \kappa_n$ and thus the kissing number could be computed using \eqref{pr:Main}, and approximations  $\gamma(\mathcal{G}_n)$ \eqref{pr:Gamma}, $\nu(\mathcal{G}_n)$ \eqref{pr:Nu} could be used to find upper bounds on $\kappa_n$. \vspace{0.5cm}

\subsection{Convergence of the hierarchies for the kissing number} \label{subsec:conv_kn}

Let $n\in \N, \ n\ge 2$. In this subsection we show that the hierarchies proposed in Section~\ref{sec:qc} provide converging upper bounds on the kissing number $\kappa_n$. We construct a kernel $M \in \interior{\copos(S^{n-1})}$ strictly feasible for Problem~\eqref{pr:Main} when $G{=}\mathcal{G}_n$, thus Theorem~\ref{thm:convergence2} applies. We use the following result by Motzkin and Straus \cite{Alpha_simplex}:
\begin{proposition}[Motzkin and Straus \cite{Alpha_simplex}] \label{prop:MinOverSim} Let $G=(V,E)$ be a finite graph with adjacency matrix $A$ and stability number $\alpha$. Then
\[\frac{1}{\alpha}=\min_{x \in \Delta^{|V|}} x \tr (A+I) x.
\]
\end{proposition}

\begin{proposition} \label{prop:ConvSphere} For the graph $\mathcal{G}_n$ we have $\gamma_r(\mathcal{G}_n) \downarrow \kappa_n$ and
$\nu_r(\mathcal{G}_n) \downarrow \kappa_n.$
\end{proposition}
\begin{proof}
From Theorem \ref{thm:convergence2} it is enough to show that Problem~\eqref{pr:Main} is strictly feasible when $G = \mathcal{G}_n$.
Let $\beta \in (\frac{1}{2},1)$. Consider the graph $\mathcal{G}_n^\beta=(S^{n-1},E_\beta), \ (x,y) \in E_\beta$ if and only if $x \tr y \in (\beta,1)$. The stability number $\alpha(\mathcal{G}_n^\beta)$ is finite as $\mathcal{G}_n^\beta$ is a compact topological packing graph. Now define an  auxiliary function $g: S^{n-1}{\times}S^{n-1} \rightarrow \R$ by
\[g(x,y)=\begin{cases}
 0,& x\tr y \in [-1,\frac{1}{2}]\\
\frac{x\tr y-0.5}{\beta -0.5}, & x\tr y \in (\frac{1}{2},\beta]\\
 1,& x\tr y \in (\beta,1].
 \end{cases}\]
Define a kernel $M(x,y)=2\alpha(\mathcal{G}_n^\beta) g(x , y)-1.$
We claim that $(M,2\alpha(\mathcal{G}_n^\beta))$ is strictly feasible for Problem~\eqref{pr:Main} with $G=\mathcal{G}_n$. By definition, $M(x,x)=2\alpha(\mathcal{G}_n^\beta) -1$ and $M(x,y)=-1$ for $(x,y) \notin E$. To show that $M \in \interior\copos(S^{n-1})$, fix $k >0$ and $u_1,\dots,u_k \in S^{n-1}$.
Let $G^u=([k],E^u)$ be the graph defined by $(i,j) \in E^u$ if and only if $u_i\tr u_j \ge \beta$. Notice that $\alpha(\mathcal{G}_n^\beta)\ge \alpha(G^u)$. Let $A^u$ be the adjacency matrix of $G^u$.  By the definition of $g$,
\begin{align}
g(u_i,u_j) \ge A^u_{ij}+I_{ij} \text{ for all } i,j \le k \label{eq:AdjIneq}
\end{align}
Let $x \in \Delta^k$, then
\begin{align*}
\sum_{i=1}^{k}\sum_{j=1}^{k}  M(u_i,u_j)x_ix_j &=  2\alpha(\mathcal{G}_n^\beta) \sum_{i=1}^{k}\sum_{j=1}^{k} g(u_i,u_j) x_ix_j - \sum_{i=1}^{k}\sum_{j=1}^{k}x_ix_j  \hspace{-5em}  \\
&\ge 2\alpha(\mathcal{G}_n^\beta)  \min_{x \in \Delta^k}x \tr (A^u+I) x-1 && \lrp{\text{by }\eqref{eq:AdjIneq}}  \\
&= \frac{2\alpha(\mathcal{G}_n^\beta) }{\alpha(G^u)}-1 && \lrp{\text{by Proposition }\ref{prop:MinOverSim}}  \\
&\ge \frac{2\alpha(\mathcal{G}_n^\beta) }{\alpha(\mathcal{G}_n^\beta)}-1\\
&=1
\end{align*}
Therefore $M \in \interior{\copos(V)}$ by Proposition~\ref{prop:interior}.
%\qed
\end{proof}

Our next goal is to compute upper bounds on $\kappa_n$ solving Problems \eqref{pr:Gamma} and \eqref{pr:Nu},  exploiting the symmetry of the graph $\mathcal{G}_n$. Proposition \ref{prop:gammaInf} below shows that, similarly to the finite case, $\gamma_r$ is never tight for $\kappa_n$ and $\gamma_r$ provides only trivial bounds for small $r$. Thus  we further restrict our attention to $\nu_r(\mathcal{G}_n)$.
\begin{proposition} \label{prop:gammaInf}
Let $n > 0, r \ge 0$. Let $\gamma_r(\mathcal{G}_n)$ be the optimal value of Problem~\eqref{pr:Gamma} with $G=\mathcal{G}_n$. Then $\gamma_r(\mathcal{G}_n)>\kappa_n$ for all $r\in \N$, and if $\gamma_r(\mathcal{G}_n)<\infty$, then $r\ge \kappa_n-1$.
\end{proposition}
\begin{proof}If Problem~\eqref{pr:Gamma} with  $G=\mathcal{G}_n$ is infeasible, then $\gamma_r(\mathcal{G}_n)= \infty >\kappa_n$. So we assume feasibility.  Let $(K, \lambda)$ be a feasible solution for Problem~\eqref{pr:Gamma} with  $G=\mathcal{G}_n$.
Let $U=\{u_1,\dots,u_{\alpha(\mathcal{G}_n)}\}\subset S^{n-1}$ be a maximum stable set of $\mathcal{G}_n$, and denote by $\mathcal{G}_n^U$  the subgraph of $\mathcal{G}_n$ induced by $U$. We have then,
\begin{equation}\label{eq:4.2KP}
  \gamma_r(\mathcal{G}_n^U) \le \gamma_r(\mathcal{G}_n) < \infty.
\end{equation}
Therefore by Theorem 4.2 from \cite{KP}, $r\ge \alpha(\mathcal{G}_n^U)-1=\alpha(\mathcal{G}_n)-1=\kappa_n-1$.  Moreover, by Lemma~\ref{lem:Csubset} $(K^U,\lambda)$ is a feasible solution to Problem\eqref{pr:Gamma} with $G=\mathcal{G}_n^U$,  and  Corollary 2 in \cite{alphaPVZ} implies
\begin{equation}\label{eq:2PVZ}
\kappa_n< \gamma_r(\mathcal{G}_n^U).
\end{equation}
Using \eqref{eq:4.2KP} we obtain $ \kappa_n< \gamma_r(\mathcal{G}_n)$.
%\qed
\end{proof} \vspace{0.5cm}
%Inspired by Corollary 2. in \cite{alphaPVZ} which shows that  for a finite graph $G$ the bound $\gamma_r(G)$ is never tight for $\alpha(G)$ if the latter is larger than one, we conjecture that the bound $\gamma_r(\mathcal{G}_n)$ is never tight for $\kappa_n$:
%
%\begin{conjecture} \label{conj:tightness}
%$\gamma_r(\mathcal{G}_n) >\kappa_n$ for all nonnegative integers $r$.
%\end{conjecture}

%\subsection{Implementation and numerical results} \label{subsec:results}

\subsection{Using symmetry of the sphere to simplify the problem} \label{subsec:fewerVars} 

To implement the bound $\nu_r(\mathcal{G}_n)$, we exploit convexity of Problem~\eqref{pr:Nu} and invariance of $\mathcal{G}_n$ under orthogonal group $O_n$ similarly to \cite{symStarAlg,copHierSym,symSDP}. Hence we only need to characterize the subset of $Q^{S^{n-1}}_r$ invariant under the action of $O_n$, which we denote by $\big(Q^{S^{n-1}}_r\big)^{O_n}$  and further we restrict our attention to the following problem:

\begin{align}
\nu_r(\mathcal{G}_n) = \inf_{K,\lambda \in \R}& \ \lambda\nonumber\\
\st \hspace{0.1cm} & \  K(x,x) = \lambda-1, &\text{for all } x \in S^{n-1},\nonumber \\
&\ K(x,y) = -1, &\text{for all } x,y \text{ with } x\tr y\in  [-1,\tfrac{1}{2}], \nonumber \\
& \ K\in \big(Q^{S^{n-1}}_r\big)^{O_n}  \nonumber
\end{align}
The condition $K\in \big(Q^{S^{n-1}}_r\big)^{O_n}$ cannot be implemented directly and requires further simplification, which we will do in two steps. First, we reduce the number of variables in the problem.  Namely, instead of $n(r+2)-$variate functions we further work with $\binom{r+2}{2}-$variate functions. Each argument of such a function  corresponds to an inner product between a pair of variables $x,y,z_1,\dots,z_r \in S^{n-1}$. As a result, the number of variables does not grow with the dimension of the sphere. Then we characterize $\big(Q^{S^{n-1}}_r\big)^{O_n}$ in terms of these new variables using the results from~\cite{PsdExtension}.

To reduce the number of variables in the problem, for $x=(x_1,\dots,x_d) \in \R^d$  let
\begin{align*}
O_n\lrp{x} =\big \{y\in \R^d: \ y_i\tr y_j =x_i\tr x_j, \ \forall i,j \in [d] \big \}
\end{align*}
be the orbit of $x$ under the action of $O_n$. The set of all orbits $\mathcal{O}=\{O_n(x): x \in (S^{n-1})^d\}$ is in natural correspondence with the following spectrahedron:
\begin{align}
\Y^{d}{=}&   \big \{Y \in \Sym^d :  \text{there is} \ x{=}(x_1,\dots,x_d) \in (S^{n-1})^d \text{ such that }  Y_{ij}{=}x_i \tr x_j \big \} \label{def:InnerProd} \\
 = &  \big \{Y \in \Sym^d :  Y_{ii}=1 \ \text{ for all } i \in [d], \  Y \succeq 0 \big \}. \nonumber
\end{align}
%Define
%\[ \Phi_d =\big \{ \phi \in C\lrp{\Y^{d}} \}. \]
Thus for any  $F \in C\lrp{(S^{n-1})^d}$ invariant under the action of $O_n$
%$\mathcal{F}_d^{O_n} \cong \Phi_d$.
there exists $\phi_F:\Y^{d}\rightarrow \mathbb{R}, \ $ such that
\begin{align}
F(x)=\phi_F(x \tr x), \text{ for all } x \in (S^{n-1})^{d}. \label{def:phi}
\end{align}

\begin{proposition}\label{prop:phi} For $d,n\in N_+, \ d\le n$, let $F \in C\lrp{(S^{n-1})^d}$ be invariant under the action of $O_n$ and let $\phi_F:\Y^{d} \rightarrow \R$ be such that $F(x) = \phi_F(x \tr x), \text{ for all } x \in (S^{n-1})^{d}$, then
\begin{enumerate}[label=\alph*.]
\item \label{Pr7.itema} $F$ is non-negative if and only if $\phi_F$ is non-negative.
\item \label{Pr7.itemb} $F$ is continuous if and only if $\phi_F$ is continuous.
\end{enumerate}
\end{proposition}
\begin{proof}
Part~\ref{Pr7.itema}  is straightforward as co-domains of $\phi_F$ and $F$ coincide. Part~\ref{Pr7.itemb} follows from Lemma 2 in \cite{PsdExtension}.
%\qed
\end{proof}
Proposition~\ref{prop:phi} and \eqref{def:phi} allow us to restrict our attention to functions $\phi:\Y^{d} \rightarrow \R$.  We use the facts that $\Y^{1}=[-1,1]$, $\phi(x\tr x)=\phi(1)$ for $\phi:\Y^{d} \rightarrow \R$ and definitions~\eqref{def:stack} and \eqref{def:Sigma} to write
\begin{align}
\nu_r(\mathcal{G}_n) = \inf_{\phi,F} \ & \ \ \phi(1)+1
\label{pr:Nu1} \\
\st \hspace{0cm}& \ \  \phi(u) \le -1, \ \text{for all } u \in  [-1,\tfrac{1}{2}], \nonumber \\
& \ \ \phi(x_1\tr x_2)+\dots+\phi(x_{r+1} \tr x_{r+2})\ge \sigma\lrp{F(x_1,\dots,x_{r+2})},  \nonumber  \\
& \ \ \text{for all } x_1,\dots x_{r+2} \in S^{n-1}, \nonumber\\
& \ \ \phi \in  C([-1,1]), \ F \in C\lrp{(S^{n-1})^{r+2}}^{O_n} \nonumber\\
& \ \ F \text{ is 2-PSD}. \nonumber
\end{align}

Our second step is to work with the last condition $F \text{ is 2-PSD}$ and characterize 2-PSD tensors on $\lrp{S^{n-1}}^{r+2}$ invariant under the action of $O_n$. Consider $r=0$ first. In this case the condition $F \in C\lrp{(S^{n-1})^{r+2}}$ is 2-PSD reduces to the condition $F$ is PSD. To implement this condition, we recall the celebrated Schoenberg's theorem about invariant PSD kernels on the unit sphere:
\begin{proposition}
[ \textbf{Schoenberg   \cite{Schoenberg}}  ] \label{thm:PSD_univ}
Let $n\ge 2$. $F \in \cK(S^{n-1})$ is invariant under the action of $O^n$ and PSD if and only if
\begin{align}
F(x,y)=\sum_{i  \ge 0} c_i P^{\tfrac{n}{2}{-}1}_i  (x \tr y ), \ c_i \ge 0, \label{eq:expJac}
\end{align}
where the series converges absolutely uniformly. And, $P^{\tfrac{n}{2}{-}1}_i$ is the  \emph{Gegenbauer polynomials} of degree $i$ and order $\tfrac{n}{2}{-}1$.
\end{proposition}
Recall that for any $\alpha$ and $d \ge 0$, the Gegenbauer polynomials $P_d^{\alpha} \in \R[t]$ of order $\alpha$ and degree $d$ can be inductively defined as $ P_0^\alpha(t)=1, P_1^\alpha(t)=2\alpha t$, and for $d>1$
\begin{align*}
dP_d^\alpha(t)=2t(d+\alpha-1)P_{d-1}^\alpha(t)-(d+2\alpha-2)P_{d-2}^\alpha(t).
% \label{Gegenbauer}
\end{align*}

Shoenberg's theorem was used in~\cite{DGS} to obtain upper bounds on the spherical codes, and the kissing number in particular. We use it when $r=0$ to substitute the condition $F \text{ is 2-PSD}$ in Problem~\eqref{pr:Nu1} with expansion~\eqref{eq:expJac}. This results in an LP with infinitely many unknowns, the coefficients $c_i, \ i\in \R_+$. This LP is equivalent to the kissing number upper bound problem in~\cite{DGS} (see Section~\ref{sec:connections}).

Now we move to the case $r\ge 1$ and work with the following extension of Shoenberg's theorem:

\begin{proposition}[Theorem 1 in \cite{PsdExtension}] \label{thm:PSD_multiv} Let $r \ge 0$ and $n \ge r+2$.
$F \in C(S^{n-1})^{r+2}$ is invariant under the action of $O_n$ and 2-PSD if and only if for every $x,y\in S^{n-1}, z\in (S^{n-1})^r$
 {\fontsize{10}{8}{\begin{align}
\hspace{-0.2cm} F(x,y,z){=}\sum_{i \ge 0} c_i(\chi, \upsilon,Z)P^{\tfrac{n-r}{2}{-}1}_i\lrp{\frac{x \tr y{-}\chi \tr Z^{-1}\upsilon}{\sqrt{(1{-}\chi \tr Z^{-1}\chi)(1{-}\upsilon \tr Z^{-1}\upsilon) }}}, \label{eq:expJacGen}
\end{align}}}
where the series converges uniformly, $\chi=z \tr x,$ $\upsilon=z\tr y$, $Z=z\tr z$, and $c_i(\chi,\upsilon,Z)$ are continuous functions such that for all $z\in (S^{n-1})^r$, $c_i(\cdot, \cdot,Z)$ are PSD.
\end{proposition}
Besides Schoenberg's theorem,  Proposition~\ref{thm:PSD_multiv} generalizes Theorem 3.2 from Bachoc and Vallentin~\cite{BV}, which is used to obtain the best currently known kissing numbers.  We use~\eqref{eq:expJacGen} to substitute the condition $F \text{ is 2-PSD}$ in Problem~\eqref{pr:Nu1} with $r\ge 1$. Now we can rewrite Problem~\eqref{pr:Nu1} given the results of Proposition~\ref{thm:PSD_multiv}:
\begin{align}
\nu_r(\mathcal{G}_n) = \inf_{\phi,F} \ & \ \ \phi(1)+1
\label{pr:Nu2} \\
\st \hspace{0cm}& \ \  \phi(u) \le -1, \ \text{for all } u \in  [-1,\tfrac{1}{2}], \nonumber \\
& \ \ \phi(x_1\tr x_2)+\dots+\phi(x_{r+1} \tr x_{r+2})- \sigma\lrp{F(x_1,\dots,x_{r+2})}\ge 0,  \label{cond:nonneg} \\
& \ \ \text{for all } x_1,\dots x_{r+2} \in S^{n-1}, \nonumber\\
& \ \ \phi \in  C([-1,1]), \ F \in C\lrp{(S^{n-1})^{r+2}} \nonumber \\
& \ \ F \text{ is of the form }\eqref{eq:expJacGen}. \nonumber
\end{align}
By restricting the last condition in this problem to a particular set of polynomials, a tractable convex problem is obtained. This procedure is described in the next subsection.\vspace{0.5cm}

\subsection{Implementation and numerical results} \label{subsec:results}

%To implement the bound $\nu_r(\mathcal{G}_n)$,
 We use Stone-Weierstrass theorem and approximate continuous functions on $\Y^{r+2}$ by polynomials of $\binom{d}{2}$ variables. Recall that each variable corresponds to an inner product between a pair of variables $x,y,z_1,\dots,z_r \in S^{n-1}$. %\par
Inspired by~\cite{BV}, we also restrict the functions $\{c_i\}_{i \in \N}$ to use in \eqref{eq:expJacGen}. For an $r$-variate vector of variables $X$ and $d\in \N$, let $m^d(X)$ be the vector of all possible monomials in variables $X$ of degrees up to $d$. Let $x,y,z_1,\dots,z_r \in S^{n-1}$. For $\chi{=}[z_1,\dots,z_r] \tr x,$ $\upsilon{=}[z_1,\dots,z_r]\tr y$, $Z=[z_1,\dots,z_r]\tr [z_1,\dots,z_r]$, define $c_i(\chi, \upsilon,Z)$ for every $i \in \N$ as follows:
\begin{align}
c_i(\chi, \upsilon,Z){:=}m^{d_i}\lrp{\chi}\tr C_i(Z) m^{d_i}\lrp{\upsilon}|Z|^i\sqrt{\lrp{1{-}\chi\tr Z^{-1} \chi}\lrp{1{-}\upsilon\tr Z^{-1} \upsilon}}^{\ i}, \label{def:shapeC}
\end{align}
where $d\in \N$, $|Z|$ denotes the determinant of the matrix $Z$, and for any $Z\in \Y^{r}$ we have $C_i(Z)\succeq 0$. \par

\begin{proposition} \label{prop:shapeC} Let $n\in \N_+$ and $r, N\in \N$. Consider $x,  y,  z_1,\dots,z_r \in S^{n-1}$. For $\chi{=}[z_1,\dots,z_r] \tr x,$ $\upsilon{=}[z_1,\dots,z_r]\tr y$, $Z=[z_1,\dots,z_r]\tr [z_1,\dots,z_r]$, consider a function $F \in C\lrp{(S^{n-1})^{r+2}}$ with representation~\eqref{eq:expJacGen} such that $c_i$ is of the form~\eqref{def:shapeC} for $i\le N$, and $c_i=0$ for all $i>N$.
Then $F$ is 2-PSD and invariant under $O_n$, and representation~\eqref{eq:expJacGen} is a polynomial in the inner products of $x,y,z_1,\dots,z_r$.
\end{proposition}

\begin{proof} By definition~\eqref{def:shapeC}, $c_i$ satisfies all the properties required in Theorem~\ref{thm:PSD_multiv}. Since the expansion is a finite series ($c_i=0$ for all $i>N$), convergence to
$F$ is uniform. Therefore $F$ is 2-PSD and invariant under $O_n$ by Theorem~\ref{thm:PSD_multiv}. Finally, the functions $c_i$ are defined in such a way that each element in the expansion is a polynomial in the inner products of $x,y,z_1,\dots,z_r\in S^{n-1}$.
%\qed
\end{proof}

For $r=1$, we have $Z=1$, therefore in case of $Q_{1}^{S^{n-1}}$, $C_i(Z)$ in~\eqref{def:shapeC} are simply scalar matrices and the condition $C_i(Z)\succeq 0$ can be addressed with any SDP solver. For $r=2$, $Z=z_1\tr z_2$ and $C_i(Z)$ in~\eqref{def:shapeC} are polynomial matrices. To implement the condition $C(Z)\succeq 0$ in this case, we use
\begin{align}
C_i(Z)=T_i^{d^{z,1}_{i}}(z_1\tr z_2)\tr T_i^{d^{z,1}_{i}}(z_1\tr z_2)+\lrp{1{-}(z_1\tr z_2)^2} T_i^{d^{z,2}_{i}}(z_1\tr z_2)\tr T_i^{d^{z,2}_{i}}(z_1\tr z_2), \label{eq:cMatr}
\end{align}
where $T_i^{d^{z,1}_{i}},T_i^{d^{z,2}_{i}}$ are polynomial matrices of degrees $d^{z,1}_{i},d^{z,2}_{i} \in \N$ respectively. Theorem 2 by Hol and Scherer~\cite{Scherer} shows that for $z_1\tr z_2 \in [-1,1]$ and $C(z_1\tr z_2)\succ 0$ representation~\eqref{eq:cMatr} always exists for some $d^{z,1}_{i},d^{z,2}_{i}$.  \par
One can reformulate representation~\eqref{eq:cMatr} via PSD matrices using, for example, Lemma 1 by Hol and Scherer~\cite{Scherer}. This results in the following function $c_i$  in case of $Q_{2}^{S^{n-1}}$, for all $i\in [N]$, $d_i, d^{z,1}_{i},d^{z,2}_{i} \in \N$:
\begin{align}
\hspace{-0.2cm}c_i(\chi, \upsilon,Z):= & \ m^{d_i}(\chi)\tr \bigg[  (M^1_{z,i}) \tr C^1_i M^1_{z,i}{+}\big(1{-}(z_1\tr z_2)^2\big) (M^2_{z,i}) \tr C^2_i M^2_{z,i} \bigg] m^{d_i}(\upsilon), \label{eq:cR2}
\end{align}
\vspace{-0.2cm} where $M^j_{z,i}=m^{d^{z,j}_{i}}(z_1\tr z_2)\otimes I_{\binom{r+d_i}{r}}$ for $j\in \{1,2\}$
%\begin{align}
%\textstyle{c_i^Z(\chi,\upsilon):=m^{d}_i(\chi_z)\tr C^1_i \  m^{d}_i(\upsilon_z)+(1-z^2)m^{d{-}1}(\chi_z)\tr C^2_i \ m^{d{-}1}(\upsilon_z),} \label{eq:cR2}
%\end{align}
and $C^1_i\succeq 0,C^2_i \succeq 0$.\vspace{0.5cm}

\newpage Next, we provide the details on the implementation of Problem~\eqref{pr:Nu2} using expansion~\eqref{eq:expJacGen} and coefficients  $c_i$ as in \eqref{eq:cR2}. For $Q_{0}^{S^{n-1}}, Q_{1}^{S^{n-1}}, Q_{2}^{S^{n-1}}$ we use Gegenbauer polynomials of degrees up to $N_0=24,N_1=12,N_2=4$, respectively. For $Q_{1}^{S^{n-1}}$ we set the degree $d_i$ in~\eqref{def:shapeC} for each $i\in [0,1,\dots,N_1]$ to $2N_1-2i$. For $Q_{2}^{S^{n-1}}$ for each $i\in [0,1,\dots,N_2]$, we set  $d^{z,1}_{i}=N_2-i$, $d^{z,2}_{i}=N_2-i-1$ and $d_i=2(N_2-i)$. \par
We set some restrictions on the structure of the problem so that the total degree of the polynomial in  $Q_{1}^{S^{n-1}}$ is $2N_1$, and the degree of each variable is at most $N_1$; the total degree of the polynomial in  $Q_{2}^{S^{n-1}}$ is $4N_2$, and the maximal degree of each variable  is $2N_2$. This structure allows us to generate fewer monomials while solving polynomial optimization problems. We also partially restrict the inequality constraint \eqref{cond:nonneg}. For $Q_{0}^{S^{n-1}}$, this inequality can be replaced by equality without loss of generality due to the shape of the objective: the constraint will hold as equality in optimum.
 For $Q_{1}^{S^{n-1}}, Q_{2}^{S^{n-1}}$, we replace the non-negativity condition in \eqref{cond:nonneg} by simpler, though stronger, conditions. For $Q_{2}^{S^{n-1}}$, we require that the left-hand side in \eqref{cond:nonneg} is a sum of squares of polynomials (SOS). For  $Q_{1}^{S^{n-1}}$, we say that the left-hand side in \eqref{cond:nonneg} equals $q_1+(1-(x\tr z_1)^2)q_2$, where $q_1,q_2$ are SOS. \par
 For $Q_{1}^{S^{n-1}}$, the structural restrictions do not influence the rounded values of the bounds, but the problem becomes substantially smaller. However, for $Q_{2}^{S^{n-1}}$, we could solve the restricted problems up to $N_2{=}7$, but the  restrictions cause feasibility issues. Therefore we present the bounds for $N_2{=}4$ only. \par
Further, we undertake several steps which make the problem smaller and more numerically stable. First, Gegenbauer polynomials are normalized so that $P^{\tfrac{n-r}{2}{-}1}_i(1)=1$ for all $i\in [N]$, this can be done without loss of generality and provides more numerically stable optimization results. Next, we exploit invariance of Problem~\eqref{pr:Nu2} under all possible permutations of the corresponding variables $x,y,z_1,z_2 \in S^{n-1}$. This allows us to decrease the number of constraints in the problem; %We ensure equality of our polynomials to SOS in \eqref{cond:nonneg} by equalizing the coefficients for all monomials on the right-hand side and on the left-hand side.
%By permutation invariance, the monomials which can be obtained from each other by appropriate permutations must have equal coefficients.
when equating polynomials to SOS in \eqref{cond:nonneg}, we only generate the constraints corresponding to the orbits of monomials.

%We also have to mention that for $r=2$, the shape of the polynomials \eqref{eq:cR2} in the constraint \eqref{cond:nonneg} is such that the resulting polynomials are invariant under some permutations of $x,y,z_1,z_2$, which could be used to simplify the SDP conditions generated by \eqref{eq:cR2}, see \cite{symStarAlg,copHierSym,symSDP}, however we do not exploit this type of invariance\juan{why}.
Table~\ref{tabFinal} shows the bounds we obtain.  The bounds from Problem~\eqref{pr:Nu2} in the table are rounded up. We do not use the exact arithmetic to check feasibility of the bounds, but we verify that all PSD constraints are satisfied, and equality constraints violations are of the order $10^{-4}$ or smaller.  The bound for $Q_{2}^{S^{n-1}}$ cannot be computed for $n\le 3$ since in this case $n<r+2$, and Proposition~\ref{thm:PSD_multiv} does not apply. \par
The bounds we could compute are not better than any of the best existing bounds, however, for $n\le 16$ the results for $Q_{1}^{S^{n-1}}$ are close to the  best bounds, while obtained in a much shorter time than in~\cite{BestBound,MV}. The main reason is that our problems are by construction smaller than the ones from~\cite{BestBound,MV}, due to the structural restrictions mentioned earlier in this section. Also, we do not use the solver SDPA-GMP  with arbitrary precision arithmetic in contrast to \cite{BestBound,MV}. The bounds obtained with SDPA-GMP can be feasible with high precision at the cost of longer running times. The running times of our problems would increase if we solved them with high precision.  \par

\begin{table}[H]
\centering
\caption{The upper bounds from Problem~\eqref{pr:Nu2} and the best existing upper and lower bounds on the kissing number. The bounds for $Q_{1}^{S^{n-1}}$ which conicide with the best bounds are marked in bold. \label{tabFinal} }
\renewcommand{\tabcolsep}{5pt}
\begin{tabular}{|c|c|c|c|c|c|}
\hline
\begin{tabular}[c]{@{}c@{}}Dimension\\ n\end{tabular} & \begin{tabular}[c]{@{}c@{}}$Q_{0}^{S^{n-1}}$ \\
 24  Gegenb. \\ polynom. \end{tabular} & \begin{tabular}[c]{@{}c@{}}$Q_{1}^{S^{n-1}}$ \\ 12 Gegenb.\\ polynom.\end{tabular} & \begin{tabular}[c]{@{}c@{}}$Q_{2}^{S^{n-1}}$ \\ 4  Gegenb. \\ polynom.\end{tabular} & \begin{tabular}[c]{@{}c@{}}Best\\ upper\\ bound\end{tabular} & \begin{tabular}[c]{@{}c@{}}Best\\ lower\\ bound \cite{LB} \end{tabular} \\ \hline
3                                                     & 13.16                                                                        & \textbf{12.54}                                         & -                                                     & 12\cite{KN3}                                                           & 12                                                           \\
4                                                     & 25.56                                                                        & \textbf{24.50}                                         & 25.77                                                 & 24                                                          \cite{KN4} & 24                                                           \\
5                                                     & 46.34                                                                        & 45.16                                                  & 47.74                                                 & 44                                                          \cite{MV}
& 40                                                           \\
6                                                     & 82.64                                                                        & \textbf{78.90}                                         & 85.93                                                 & 78 \cite{BV}
& 72                                                           \\
7                                                     & 140.17                                                                       & 136.30                                                 & 154.90                                                & 134 \cite{MV}
& 126                                                          \\
8                                                     & 240.00                                                                       & \textbf{240.00}                                        & 297.21                                                & 240 \cite{KN824,KN824_2}
& 240                                                          \\
9                                                     & 380.10                                                                       & 371.75                                                 & 724.69                                                & 364 \cite{BestBound}
 & 306                                                          \\
10                                                    & 595.83                                                                       & 580.68                                                 & 4,525.45                                              & 554                                                         \cite{BestBound}  & 500                                                          \\
11                                                    & 915.39                                                                       & 899.68                                                 & infeasible                                            & 870                                                         \cite{BestBound}  & 582                                                          \\
12                                                    & 1,416.10                                                                     & 1,384.68                                               & infeasible                                            & 1,357                                                       \cite{BestBound}  & 840                                                          \\
13                                                    & 2,233.64                                                                     & 2,152.53                                               & infeasible                                            & 2,069                                                       \cite{BestBound}  & 1,154                                                        \\
14                                                    & 3,492.22                                                                     & 3,307.45                                               & infeasible                                            & 3,183                                                       \cite{BestBound}  & 1,606                                                        \\
15                                                    & 5,431.03                                                                     & 5,043.03                                               & infeasible                                            & 4,866                                                       \cite{BestBound}  & 2,564                                                        \\
16                                                    & 8,313.79                                                                     & 7,863.00                                               & infeasible                                            & 7,355                                                       \cite{BestBound}  & 4,320                                                        \\
17                                                    & 12,218.68                                                                    & 12,050.45                                              & infeasible                                            & 11,072                                                      \cite{BestBound}  & 5,346                                                        \\
18                                                    & 17,877.07                                                                    & 18,008.00                                              & infeasible                                            & 16,572                                                      \cite{BestBound}  & 7,398                                                        \\
19                                                    & 25,900.79                                                                    & 26,672.27                                              & infeasible                                            & 24,812                                                      \cite{BestBound}  & 10,668                                                       \\
20                                                    & 37,974.01                                                                    & 39,554.23                                              & infeasible                                            & 36,764                                                      \cite{BestBound}  & 17,400                                                       \\
21                                                    & 56,851.69                                                                    & 59,458.13                                              & infeasible                                            & 54,584                                                      \cite{BestBound}  & 27,720                                                       \\
22                                                    & 86,537.49                                                                    & 88,326.01                                              & infeasible                                            & 82,340                                                      \cite{BestBound}  & 49,896                                                       \\
23                                                    & 128,095.86                                                                   & 130,270.28                                             & infeasible                                            & 124,416                                                     \cite{BestBound}  & 93,150                                                       \\
24                                                    & 196,560.00                                                                   & \textbf{196,560.02}                                    & infeasible                                            & 196,560                                                     \cite{KN824,KN824_2} & 196,560                                                      \\
25                                                    & 278,364.38                                                                   & 282,690.33                                             & infeasible                                            & 278,083                                                     \cite{Pfender}  & 197,040                                                      \\
26                                                    & 396,977.00                                                                   & 403,772.00                                             & infeasible                                            & 396,447                                                     \cite{Pfender}  & 198,480                                                      \\
30                                                    & 1,653,914.18                                                                 & 1,749,936.18                                           & infeasible                                            &                                                              & 219,008                                                      \\
35                                                    & 10,510,137.84                                                                & 13,835,411.99                                          & infeasible                                            &                                                              & 370,892                                                      \\
40                                                    & infeasible                                                                   & infeasible                                             & infeasible                                            &                                                              & 1,063,216\\
\hline
\begin{tabular}[c]{@{}c@{}}Approx.\\ solution\\ time \end{tabular} & $\le 1$ sec. & $10$ sec. & $25$ sec. & \begin{tabular}[c]{@{}c@{}}\cite{BestBound}: $12$ hours,\\\cite{MV}: ${\ge}1$ week \end{tabular} & -  \\
\hline
\end{tabular}
\end{table}
\newpage Since the sum of copositive kernels is copositive, we could set the kernel in Problem~\eqref{pr:Nu2} to be the sum of the functions from  $Q_{0}^{S^{n-1}}, Q_{1}^{S^{n-1}}$ and $Q_{2}^{S^{n-1}}$. Table~\ref{tabFinal} suggests that this could provide potentially stronger bounds, but the resulting problems are numerically unstable and do not improve on the bounds from Table~\ref{tabFinal}. Nevertheless, using this approach we can compare the optimization problems for our bounds and the existing SDP upper bound used in \cite{BV,BestBound,MV}. This is done in the next section.

To conclude this section, we recall that the kissing number problem is a particular case of the spherical codes problem. As it was mentioned before, in the spherical codes problem we are interested in the maximum number of points on the unit sphere in $\R^n$  for which the pairwise angular distance is not smaller than some value $\theta$. The kissing number problem corresponds to $\theta=\tfrac{\pi}{3}$. Figure~\ref{fig:bounds}~shows how the bounds from  $Q_{0}^{S^{n-1}}$ and $Q_{1}^{S^{n-1}}$ change when $\theta$ changes. For a more informative comparison, we add lower bounds computed using the algorithm by Roebers~\cite{Roebers}.  
\input{kissingPlot1}
We have examined $3\le n \le 8$, and the results exhibit the same pattern for each $n$. We choose $n\in \{7,8\}$ since for these dimensions the difference in performance is most visible.  The choice of $\theta$ is motivated by the lower bound algorithm in~\cite{Roebers}: for a given number of points on the unit sphere (which is equal to the lower bound in our case), the algorithm finds a feasible allocation of the points with the minimum distance $\theta$. We run the algorithm for $30$ min. using the lower bounds $[20,40,60,\dots,400]$ to obtain $\theta$ for each of these lower bounds. Next, we solve Problem~\eqref{pr:Nu2} replacing $\tfrac{1}{2}$ in the first constraint with the corresponding $\cos \theta$. We present the bounds for which all PSD constraints are satisfied, and equality constraints violations are of the order $10^{-4}$ or smaller. \vspace{0.5cm}

\section{Connection to the existing upper bound approaches} \label{sec:connections}

In this section we provide some remarks on the connection between our relaxations and the existing LP bound by Delsarte, Goethals and Seidel \cite{DGS} and SDP bound by Bachoc and Vallentin \cite{BV}. We start by showing that the LP bound equals our bound for $Q_0^{S^{n-1}}$. As explained in Secion~\ref{subsec:results}, inequality \eqref{cond:nonneg} can be replaced by equality, thus, our formulation for $Q_0^{S^{n-1}}$ is equivalent to
{\fontsize{10}{6}{\begin{align*}
%\nu_r(\mathcal{G}_n) \le
\inf & \ \sum_{k=0}^{N_0} a_k+1\\
%\nonumber \\
\st &  \ \sum_{k=0}^{N_0} a_k P^{\tfrac{n}{2}{-}1}_i(u) \le -1 , \ \text{for all } u \in  [-1,\tfrac{1}{2}], \nonumber \\
& \ a_k \ge 0, \ \text{for all } k \in \{0,\dots,N_0\}. %\nonumber
\end{align*}}}

From the shape of the objective, using $P^{\tfrac{n}{2}{-}1}_0(u)=1$, it is clear that in optimality $a_0=0$. Therefore, our bound from $Q_0^{S^{n-1}}$ coincides with the LP bound by Delsarte, Goethals and Seidel \cite{DGS}.

Now let us consider the SDP bound. This bound is formulated in Theorem 4.2 from~\cite{BV}:
%{\fontsize{10}{6}{\begin{align}
%\alpha(\mathcal{G}_n) \le  \inf & \ \sum_{k=1}^{d} a_k+b_{11}+\ip{F_0}{S^n_0(1,1,1)}
%+1
%\label{pr:Nu4} \\
%\st \hspace{0.8cm}& \ \sum_{k=1}^{d} a_k P^{\tfrac{n}{2}{-}1}_i(u){+}2b_{12}{+}b_{22}{+}3\sum_{k=0}^{d}\ip{F_k}{S^n_k(u,u,1)} \le -1 , \label{cond:univ} \\
%& \hspace{5.7cm} \text{for all } u \in  [-1,\tfrac{1}{2}] \nonumber \\
%& \ a_k \ge 0, \ \text{for all } k \in \{0,\dots,d\}, \nonumber  \\
%& \ F_k \succeq 0, \ \text{for all } k \in \{0,\dots,d\},\label{constrF} \\
%& \ \begin{bmatrix}
%b_{11}&b_{12}\\
%b_{12}&b_{22}
%\end{bmatrix}\succeq 0, \label{constrB}\\
%& \ b_{22}{+}\sum_{k=0}^{d}\ip{F_k}{S^n_k(u,v,t)} \le 0 , \nonumber  \\
%& \hspace{1.7cm}  \text{for all } u,v,t \in  [-1,\tfrac{1}{2}], 1{+}2uvt{-}u^2{-}v^2{-}t^2\ge 0 . \label{constrNonNeg}
%\end{align}}}
{\fontsize{10}{6}{\begin{align}
\hspace{1cm} \alpha(\mathcal{G}_n) \le  \inf & \ \sum_{k=1}^{d} a_k+b_{11}+\ip{F_0}{S^n_0(1,1,1)}
+1
\label{pr:Nu4} \\
\st \hspace{0cm}& \ \sum_{k=1}^{d} a_k P^{\tfrac{n}{2}{-}1}_i(u){+}2b_{12}{+}b_{22}{+}3\sum_{k=0}^{d}\ip{F_k}{S^n_k(u,u,1)} \le -1 , \label{cond:univ} \\
& \hspace{5.7cm} \text{for all } u \in  [-1,\tfrac{1}{2}] \nonumber \\
& \ a_k \ge 0, \ \text{for all } k \in \{0,\dots,d\}, \nonumber  
%& \ F_k \succeq 0, \ \text{for all } k \in \{0,\dots,d\},\label{constrF} \\
%& \ \begin{bmatrix}
%b_{11}&b_{12}\\
%b_{12}&b_{22}
%\end{bmatrix}\succeq 0, \label{constrB}\\
%& \ b_{22}{+}\sum_{k=0}^{d}\ip{F_k}{S^n_k(u,v,t)} \le 0 , \nonumber  \\
%& \hspace{1.7cm}  \text{for all } u,v,t \in  [-1,\tfrac{1}{2}], 1{+}2uvt{-}u^2{-}v^2{-}t^2\ge 0 . \label{constrNonNeg}
\end{align}}}
{\fontsize{10}{6}{\begin{align}
%\alpha(\mathcal{G}_n) \le  \inf & \ \sum_{k=1}^{d} a_k+b_{11}+\ip{F_0}{S^n_0(1,1,1)}
%+1
%\label{pr:Nu4} \\
%\st \hspace{0.8cm}& \ \sum_{k=1}^{d} a_k P^{\tfrac{n}{2}{-}1}_i(u){+}2b_{12}{+}b_{22}{+}3\sum_{k=0}^{d}\ip{F_k}{S^n_k(u,u,1)} \le -1 , \label{cond:univ} \\
%& \hspace{5.7cm} \text{for all } u \in  [-1,\tfrac{1}{2}] \nonumber \\
%& \ a_k \ge 0, \ \text{for all } k \in \{0,\dots,d\}, \nonumber  \\
\hspace{3cm}& \ F_k \succeq 0, \ \text{for all } k \in \{0,\dots,d\},\label{constrF} \\
& \ \begin{bmatrix}
b_{11}&b_{12}\\
b_{12}&b_{22}
\end{bmatrix}\succeq 0, \label{constrB}\\
& \ b_{22}{+}\sum_{k=0}^{d}\ip{F_k}{S^n_k(u,v,t)} \le 0 , \nonumber  \\
& \hspace{1.7cm}  \text{for all } u,v,t \in  [-1,\tfrac{1}{2}], 1{+}2uvt{-}u^2{-}v^2{-}t^2\ge 0 . \label{constrNonNeg}
\end{align}}}
Using the results of the analysis of the bound for $Q_0^{S^{n-1}}$ from the previous paragraph and the fact that the sum of copositive kernels is copositive, we can reformulate Problem~\eqref{pr:Nu2} for the sum of $Q_0^{S^{n-1}}$ and  $Q_1^{S^{n-1}}$ of degree $d$:
{\fontsize{10}{6}{\begin{align}
\hspace{1cm}\alpha(\mathcal{G}_n) \le  \inf & \ \sum_{k=1}^{d} a_k+g(1)+1
\label{pr:Nu5} \\
\st \hspace{0cm}&  \ \sum_{k=1}^{d} a_k P^{\tfrac{n}{2}{-}1}_i(u)+g(u) \le -1 , \ \text{for all } u \in  [-1,\tfrac{1}{2}], \label{cond:univ1} \\
& \ a_k \ge 0, \ \text{for all } k \in \{0,\dots,d\}, \nonumber  \\
& \ \sigma\lrp{F_g(u,v,t)}{-}g(u){-}g(v){-}g(t)\le 0, \nonumber \\
& \hspace{2.2cm}  \text{for all } u,v,t \in  [-1,1], 1{+}2uvt{-}u^2{-}v^2{-}t^2\ge 0 \label{cond:nonneg4} \\
& \ F_g \text{ is as in Proposition }\ref{prop:shapeC} \text{ with } r{=}1\text{ and }N=d.\label{constrF1}
\end{align}}}
\makebox[\linewidth][s]{In \cite{BV} a different notation is used, but, up to constant multipliers, the expression} \\ $\sum_{k=0}^{d}\ip{F_k}{S^n_k(u,v,t)}$ in constraint \eqref{constrNonNeg} coincides with $\sigma\lrp{F_g(u,v,t)}$ in constraint~\eqref{cond:nonneg4}, the expression $\sum_{k=0}^{d}\ip{F_k}{S^n_k(u,u,1)}$ in constraint~\eqref{cond:univ} coincides with $g(u)$ in constraint~\eqref{cond:univ1} and the expression $\sum_{k=0}^{d}\ip{F_k}{S^n_k(1,1,1)}$ in the objective of Problem~\eqref{pr:Nu4} coincides with $g(1)$ in the objective of Problem~\eqref{pr:Nu5}. Also, constraint \eqref{constrF} is captured by constraint~\eqref{constrF1}.\par
Now look at the differences. The first difference between Problem~\eqref{pr:Nu4} and our problem comes from the fact that we partially restrict the non-negativity constraint \eqref{cond:nonneg4}, as it was explained in Secion~\ref{subsec:results}. Therefore our bounds might be weaker. The second difference comes from the fact that we require constraint \eqref{cond:nonneg4} to hold for  a larger set of variables than Bachoc and Vallentin~\cite{BV} do in the corresponding constraint \eqref{constrNonNeg}. In the latter constraint only those combinations of vertices are considered which form a stable set in $\mathcal{G}_n$. This could make Problem~\eqref{pr:Nu4} stronger as well.
The final difference are the variables $b_{11},b_{12},b_{22}$ and the corresponding constraint \eqref{constrB}. Those appear naturally from the dual of Problem~\eqref{pr:Nu4}, but do not fit naturally into our framework. At the same time, the expression $({-}g(u){-}g(v){-}g(t))$ on the left-hand side of our constraint \eqref{cond:nonneg4} does not appear in the corresponding constraint  \eqref{constrNonNeg}, but the relation between the variables $b_{11},b_{12},b_{22}$ and $({-}g(u){-}g(v){-}g(t))$ is not straightforward. \par
As a result of the mentioned differences, constraint~\eqref{cond:univ} is  weaker than constraint~\eqref{cond:univ1}. Moreover, the combination of constraints \eqref{constrF}--\eqref{constrNonNeg}  is likely to be weaker than the combination of constraints  \eqref{cond:nonneg4} and~\eqref{constrF1}. Hence we expect that our problem  with the sum of $Q_0^{S^{n-1}}$ and  $Q_1^{S^{n-1}}$ of degree $d$ cannot provide better bounds than Problem\eqref{pr:Nu4} from~\cite{BV} of the same degree.\par
Problem~\eqref{pr:Nu2} where the sum of $Q_0^{S^{n-1}}$, $Q_1^{S^{n-1}}$ and $Q_2^{S^{n-1}}$ is used could potentially provide stronger bounds than Problem~\eqref{pr:Nu4} by Bachoc and Vallentin~\cite{BV}, but those problems are not immediately comparable. For future, it would be interesting to see whether one can combine our approach with the currently best approaches for the kissing number problem by Bachoc and Vallentin~\cite{BV} and Pfender~\cite{Pfender}.
%\section{Conclusion} \label{sec:conclusion}

\section{Proof of Proposition~\ref{prop:QrTensor}} \label{sec:proofs}\label{subsec:TensorProofs1}
We first introduce another tensor operator, a projection operator. For any $d' \le d$, and $v \in V^{d-d'}$ we define the \emph{$d'$-slice},  $\slice^{\pmb{v}}: \cT^V_{d} \rightarrow \cT^V_{d'}$ by
\begin{align}
\slice^{\pmb{v}}(T)(u): = T(u,v) \ \text{for all } T \in  \cT^V_{d},\,  u \in V^d. \label{def:slice}
\end{align}
That is, the slice is obtained by fixing all but the last indices of the tensor at $v$. We are particularly interested in 2-slices to project high order tensors to kernels.

Next, with a tensor $T \in \cT_d^{[n]}$,  we associate a degree $d$ homogeneous polynomial in $n$ variables:
\[T[x] := \sum_{v \in [n]^d} T(v)\prod_{i=1}^d x_{v_i}.\]
For $T \in \cT_d^{[n]}$, a monomial in $T[x]$ is denoted by $x_1^{\beta_1} \dotsm x^\beta_n=x^{\beta}$ where $\beta \in \N^{n}$ and $\sum_{i=1}^n \beta_i=d$. Each $x^\beta$ corresponds to $\prod_{i=1}^{d}x_{v_{i}}$ for some $v\in [n]^d$. Permuting the elements of $v$ does not change the corresponding monomial. Let
\begin{align}
\cP_\beta=\{v\in [n]^d: \prod_{i=1}^{d}x_{v_{i}}=x^\beta \}. \label{def:PiBeta}
\end{align}
Then the coefficient of the monomial $x^\beta$ in $T[x]$ is $\sum_{v\in \cP_\beta}T(v)$. We use several properties of $T[x]$ which are described using the following two lemmas:
\begin{lemma} \label{lem:PiBeta} Let $n,  d\in \N_+$ and $T \in \cT_{d}^{[n]}$. Let $x^\beta$ be a monomial in $T[x]$. Then for any $u,v\in \cP_\beta$ and $\pi \in Sym(d)$
\begin{enumerate}[label=\alph*.]
\item  $\pi v \in \cP_\beta.$
\item $\sigma(T)(v)=\sigma(T)(u).$
\end{enumerate}
\end{lemma}
%Notice that $\cP_\beta$ is invariant under the action of $Sym(r)$ (QUESTION) and
%\begin{align}
%\sigma(T)(v)=\sigma(T)(u) \text{ for any } u,v \in \cP_\beta. \label{prop:PiBeta}
%\end{align}
\begin{lemma}\label{lem:Basic} Let $n,  d\in \N_+$.
\begin{enumerate}[label=\alph*.]
\item  For $T \in \cT_{d+2}^{[n]}$, $T[x] = \sum_{v \in [n]^d } \slice^{\pmb{v}}(T)[x] \prod_{i=1}^{d}  x_{v_i} $
\item  For any $r\in \N$ and $T \in \cT_{d}^{[n]}$: $\displaystyle \stack^\mathbf{r}(T)[x] = (e\tr x)^r T[x]$
\item  For any $T, S \in \cT_{d}^{[n]}$: $\displaystyle S[x] = T[x]$ if and only if $\sigma(S) = \sigma(T)$.
\end{enumerate}
\end{lemma}

\begin{proof} $a.$ Let $T \in \cT_{d+2}^{|n|} $. From \eqref{def:slice},
\begin{align}
\notag
\hspace{1cm} T[x] = & \ \sum_{u \in [n]^{d+2}} T(u)\prod_{i=1}^{d+2} x_{u_i}=\sum_{v \in [n]^{d}} \sum_{u_1,u_2 \in [n]}  T(u_1,u_2,v)x_{u_1}x_{u_2}\prod_{i=1}^{d} x_{v_i} \\
\notag
= & \ \sum_{v \in [n]^d }    \slice^{\pmb{v}}(T)[x] \prod_{i=1}^{d}  x_{v_i}
\end{align}
$b.$ Let $T\in \cT_{d}^{[n]}$. First, consider the case $r=0$:
\[\stack^\mathbf{r}(T)[x] =T[x]= (e\tr x)^0 T[x].\]
Now we show that $\stack^\mathbf{1}(T)[x]=(e\tr x)T[x]$ using $a.$, then, by  Lemma \ref{lem:SigmaStack} $c.$, the statement follows by induction.
%$$T^{\oplus r}[x] =T^{\oplus }[x]= \sum_{u\in [n]^{d+1}}T^{\oplus }(u)\prod_{i=1}^{d+1} x_{u_i}=\sum_{u_{d+1} \in [n]}x_{u_{d+1}} \sum_{u_1,\dots,u_{d}\in [n]}T(u_1\dotsu_{d})\prod_{i=1}^{d} x_{u_i}=(e\tr x) T[x].$$
\begin{align*}
\stack^\mathbf{1}(T)[x]=& \ \sum_{u\in [n]^{d+1}} \stack^\mathbf{1}(T)(u) \prod_{i=1}^{d+1}x_{u_{i}}\\
=& \ \sum_{u_{d+1}\in [n]} x_{u_{d+1}}  {\sum_{u_1,\dots,u_{d}\in [n]}} \stack^\mathbf{1}(T)(u) \prod_{i=1}^{d}x_{u_{i}} =(e\tr x)T[x]. \hspace{1cm}
\end{align*}
$c.$ Let $T, S \in \cT_{d}^{[n]}$.  For any $\beta \in \N^{n}$ such that $\sum_{i=1}^n \beta_i=d$, let $\cP_\beta$ be defined as in \eqref{def:PiBeta}. Then by Lemma \ref{lem:SigmaStack} $a.$ and Lemma \ref{lem:PiBeta}, for any $v\in \cP_\beta$
\begin{align*}
\sigma(T)(v)=&\frac{1}{|\cP_\beta|}\sum_{u\in \cP_\beta}  \sigma(T)(u)=\frac{1}{|\cP_\beta|} \sigma\lrp{\sum_{u\in \cP_\beta}  T(u)} =  \frac{1}{|\cP_\beta|}\sum_{u\in \cP_\beta} T( u).
\end{align*}
Recall that $\sum_{u\in \cP_\beta} T(u)$ is the coefficient of $x^\beta$, which concludes the proof.
%\begin{align*}
%\sum_{v\in \cP_\beta} T( v ) = &\sum_{v\in \cP_\beta}T(\pi  v ) =\frac{1}{d!}\sum_{\pi \in perm(d)}\sum_{ v\in \cP_\beta}T(\pi v )=\frac{1}{d!}\sum_{ v\in \cP_\beta}\sigma(T) (v )\\
%=& \sigma(T) (u_1,\dots,u_d) = \sigma(S) (u_1,\dots,u_d) =\sum_{v\in \cP_\beta} S( v ).
%\end{align*}
%Hence $T[x]=S[x]$.
%\qed
\end{proof}

\noindent  Now we are ready to prove Proposition~\ref{prop:QrTensor}.

\begingroup
\renewcommand*{\proofname}{\textbf{Proof of Proposition~\ref{prop:QrTensor}}}
\begin{proof}
Define $\cP_{\beta}$ as in \eqref{def:PiBeta}.  First, for $M \in \cK([n])$, let $\sigma \lrp{\stack^\mathbf{r}(M)}=\sigma(N)+\sigma(S) $ which implies $(e\tr x)^rM[x]=N[x]+S[x]$ by Lemma \ref{lem:Basic}. Using Lemma \ref{lem:Basic} again, the latter equality can be rewritten as
%For $T \in \cT_d^n$, from Lemma 1 and \eqref{eq:InvOplus},
\begin{align*}
(e\tr x)^rM[x]=&\sum_{v \in [n]^r } \slice^{\pmb{v}}(N)[x] \prod_{i=1}^{r} x_{v_i} +\sum_{v \in [n]^r } \slice^{\pmb{v}}(S)[x] \prod_{i=1}^{r} x_{v_i} \\\
= & \sum_{|\beta|=r}  x^{\beta} \sum_{v \in \cP_{\beta}} \slice^{\pmb{v}}(N)[x] +\sum_{|\beta|=r}  x^{\beta} \sum_{v \in \cP_{\beta}} \slice^{\pmb{v}}(S)[x]\\
= & \sum_{|\beta|=r}  x^{\beta}\hat{N}_\beta[x] +\sum_{|\beta|=r}  x^{\beta} \hat{S}_\beta[x]
%= & \sum_{|\beta|=r} x^\beta x\tr N_\beta x + \sum_{|\beta|=r} x^\beta x\tr S_\beta x,
\end{align*}
where $ \hat{N}_\beta \ge 0$ and $\hat{S}_\beta$ is PSD for all possible $\beta$ as sums of nonnegative and positive semidefinite kernels respectively. Thus $M \in Q_r^n$ or $M \in \cC_r^n$ if $S$ is the the matrix of all zeros.  \par

\noindent Now let $M \in Q_r^n$. By Lemma \ref{lem:Basic} this implies
\begin{align*}
(e\tr x)^r M[x]=& \ \stack^\mathbf{r}(M)[x]=\sum_{|\beta|=r} x^\beta N_\beta[x] + \sum_{|\beta|=r} x^\beta S_\beta[x]\\
= & \  \sum_{|\beta|=r} x^\beta \sum_{v \in \cP_{\beta}}\frac{1}{|\cP_{\beta}|} N_\beta[x] + \sum_{|\beta|=r} x^\beta \sum_{v \in \cP_{\beta}}\frac{1}{|\cP_{\beta}|} S_\beta[x].
%=& \ \sum_{v \in [n]^r} \frac{1}{|\cP_{\beta}|} N_\beta[x]\prod_{i=1}^{r} x_{v_i}  + \sum_{v \in [n]^r} \frac{1}{|\cP_{\beta}|} S_\beta[x]\prod_{i=1}^{r} x_{v_i}.
\end{align*}
\noindent Define $\hat{N}$ and $\hat{S}$ as follows:
\begin{align*}
\hat{N}:& \ \slice^{\pmb{v}}(\hat{N})=\frac{1}{|\cP_{\beta}|} N_\beta \ \text{ for all }  \beta \in \mathbb{N}^n, \ |\beta|=r \text{ and } v \in \cP_{\beta}, \\
  \hat{S}: & \ \slice^{\pmb{v}}(\hat{S})=\frac{1}{|\cP_{\beta}|}  S_\beta \ \text{ for all } \beta \in \mathbb{N}^n, \ |\beta|=r \text{ and } v \in \cP_{\beta}.
\end{align*}
Then
\begin{align*}
\stack^\mathbf{r}(M)[x] {=}\sum_{v \in [n]^r} \slice^{\pmb{v}}(\hat{N})[x]\prod_{i=1}^{r} x_{v_i}{+}\sum_{v \in [n]^r} \slice^{\pmb{v}}(\hat{S})[x]\prod_{i=1}^{r} x_{v_i}= \hat{N}[x]{+}\hat{S}[x],
\end{align*}
where $ \hat{N} \in \cN_{r+2}^{[n]}, \ S \in \cT_{r+2}^{[n]}, \ \text{2-PSD}$.
The last equality follows from Lemma \ref{lem:Basic} $a$. Hence, by the same lemma,
$ \sigma \lrp{\stack^\mathbf{r}(M)}=\sigma( \hat{N})+\sigma( \hat{S})$.
Analogously, for $M \in \cC_r^n$, we obtain  $ \sigma \lrp{\stack^\mathbf{r}(M)}=\sigma( \hat{N}), \  N \in \cN_{r+2}^{[n]}$.
%\qed
\end{proof}
\endgroup

\section*{Acknowledgements} \label{sec:ack}

We would like to thank Renata Sotirov for her useful comments on an early version of the manuscript, Miriam D\"ur and Claudia Adams for their remarks on the sufficiency condition in Proposition~\ref{prop:interior}, Fernando de Oliveira Filho for his advice on possible improvements of the bounds, and Lorenz Roebers for providing his code to compute the lower bounds in Section~\ref{subsec:results}.

\bibliographystyle{plain}
\bibliography{references_kissing}

\begin{thebibliography}{10}

\bibitem{Dels_25}
V.~V. Arestov and A.~G. Babenko.
\newblock On {D}elsarte scheme of estimating the contact num- bers.
\newblock {\em Proc. of the Steklov Inst. of Math.}, 219:36--65, 1997.

\bibitem{LbChr}
C.~Bachoc, G.~Nebe, F.M. de~Oliveira~Filho, and F.~Vallentin.
\newblock Lower bounds for measurable chromatic numbers.
\newblock {\em Geometric and Functional Analysis}, 19(3):645--661, 2009.

\bibitem{BV}
C.~Bachoc and F.~Vallentin.
\newblock New upper bounds for kissing numbers from semidefinite programming.
\newblock {\em J. Amer. Math. Soc.}, 21(3):909--924, 2008.

\bibitem{Bai16}
L.~Bai, J.E Mitchell, and J.-S. Pang.
\newblock On conic qpccs, conic qcqps and completely positive programs.
\newblock {\em Math. Program.}, 159(1-2):109--136, 2016.

\bibitem{FinPSD}
S.~Bochner.
\newblock Hilbert distances and positive definite functions.
\newblock {\em Ann. of Math.}, 42:647--656, 1941.

\bibitem{BomzeSurvey}
I.~M. Bomze.
\newblock Copositive optimization -- recent developments and applications.
\newblock {\em European Journal of Operational Research}, 216(3):509--520,
  2012.

\bibitem{SimplPart}
S.~Bundfuss and M.~D\"ur.
\newblock Algorithmic copositivity detection by simplicial partition.
\newblock {\em Linear Algebra and its Applications}, 428:1511--1523, 2008.

\bibitem{Burer09}
S.~Burer.
\newblock On the copositive representation of binary and continuous nonconvex
  quadratic programs.
\newblock {\em Math. Program.}, 120(2):479--495, 2009.

\bibitem{Burer12}
S.~Burer and H.~Dong.
\newblock Representing quadratically constrained quadratic programs as
  generalized copositive programs.
\newblock {\em Operations Research Letters}, 40(3):203--206, 2012.

\bibitem{KP}
E.~de~Klerk and D.~Pasechnik.
\newblock Approximation of the stability number of a graph via copositive
  programming.
\newblock {\em SIAM J. Optim.}, 12:875--892, 2002.

\bibitem{symStarAlg}
E.~de~Klerk, D.V. Pasechnik, and A.~Schrijver.
\newblock Reduction of symmetric semidefinite programs using the regular
  $\ast$-representation.
\newblock {\em Math. Program.}, 109(2):613--624, 2007.

\bibitem{Packing_thesis}
D.~de~Laat.
\newblock Moment methods in extremal geometry.
\newblock {\em Ph.D. thesis, Delft University of Technology}, 2016.

\bibitem{Packing2}
D.~de~Laat, F.~M. de~Oliveira~Filho, and F.~Vallentin.
\newblock Upper bounds for packings of spheres of several radii.
\newblock {\em Forum of Mathematics}, Sigma 2, e23, 2014.

\bibitem{Packing}
D.~de~Laat and F.~Vallentin.
\newblock A semidefinite programming hierarchy for packing problems in discrete
  geometry.
\newblock {\em Math. Program.}, 151(2):529--553, 2015.

\bibitem{Positive_thesis}
E.~DeCorte.
\newblock The eigenvalue method for extremal problems on infinite
  vertex-transitive graphs.
\newblock {\em Ph.D. thesis, Delft University of Technology}, 2015.

\bibitem{DistAvoid}
E.~DeCorte, F.~M. de~Oliveira~Filho, and F.~Vallentin.
\newblock Complete positivity and distance-avoiding sets.
\newblock \url{https://arxiv.org/abs/1804.09099}, 2018.

\bibitem{DGS}
P.~Delsarte, J.M. Goethals, and J.J. Seidel.
\newblock Spherical codes and designs.
\newblock {\em Geom. Dedicata}, 6:363--388, 1977.

\bibitem{alphaInf}
C.~Dobre, M.~Dur, L.~Frerick, and F.~Vallentin.
\newblock A copositive formulation for the stability number of infinite graphs.
\newblock {\em Math. Program.}, 160(1):65--83, 2016.

\bibitem{copHierSym}
C.~Dobre and J.C. Vera.
\newblock Exploiting symmetry in copositive programs via semidefinite
  hierarchies.
\newblock {\em Math. Program.}, 151(2):659--680, 2015.

\bibitem{TensorQ}
H.~Dong.
\newblock Symmetric tensor approximation hierarchies for the completely
  positive cone.
\newblock {\em SIAM J. Optim.}, 23(3):1850--1866, 2013.

\bibitem{QuadAsRef}
I.~Dukanovic and F.~Rendl.
\newblock Copositive and semidefinite relaxations of the quadratic assignment
  problem.
\newblock {\em Math. Program.}, 121:249--268, 2010.

\bibitem{DuerSurvey}
M.~D\"ur.
\newblock Copositive programming -- a survey.
\newblock {\em University of Groningen, Johann Bernoulli Institute for
  Mathematics and Computer Science, 2010. ISBN 9783642125973. Relation:
  http://www.rug.nl/informatica/onderzoek/bernoulli Rights: University of
  Groningen, Johann Bernoulli Institute for Mathematics and Computer Science.},
  2010.

\bibitem{symSDP}
K.~Gatermann and P.~A. Parrilo.
\newblock Symmetry groups, semidefinite programs, and sums of squares.
\newblock {\em J. Pure Appl. Algebra}, 192(1):95--128, 2004.

\bibitem{ChromRef}
N.~Gvozdenovi{\'c} and M.~Laurent.
\newblock The operator $\psi$ for the chromatic number of a graph.
\newblock {\em SIAM J. Optim.}, 19:572--591, 2008.

\bibitem{HardLP88}
G.~Hardy, J.~Littlewood, and G.~P\'olya.
\newblock {\em Inequalities}.
\newblock Cambridge University Press, New York, second edition, 1988.

\bibitem{PsdExtension}
O.~Kuryatnikova and J.~C. Vera.
\newblock Generalizations of {S}choenberg's theorem on positive definite
  kernels.
\newblock {\em Working paper, Tilburg University}, 2018.

\bibitem{Outer_Las}
J.~B. Lasserre.
\newblock New approximations for the cone of copositive matrices and its dual.
\newblock {\em Math. Program.}, 144(1):265--276, 2014.

\bibitem{KN824_2}
V.I. Levenshtein.
\newblock On bounds for packing in n-dimensional euclidean space.
\newblock {\em Soviet Math. Dokl.}, 20:417--421, 1979.

\bibitem{Lovasz}
L.~Lov{\'a}sz.
\newblock On the shannon capacity of a graph.
\newblock {\em IEEE Trans. Inform. Theory}, 25(1):1--7, 1979.

\bibitem{BestBound}
F.~C. Machado and F.~M. de~Oliveira~Filho.
\newblock Improving the semidefinite programming bound for the kissing number
  by exploiting polynomial symmetry.
\newblock {\em Experimental Mathematics}, 0:1--8, 2017.

\bibitem{MV}
H.~D. Mittelmann and F.~Vallentin.
\newblock High-accuracy semidefinite programming bounds for kissing numbers.
\newblock {\em Experimental Mathematics}, 19(2):175--179, 2010.

\bibitem{Alpha_simplex}
T.~S. Motzkin and E.~G. Straus.
\newblock Maxima for graphs and a new proof of a theorem of tur\'{a}n.
\newblock {\em Canad. J. Math.}, 17:533--540, 1965.

\bibitem{MultPsd}
O.~R. Musin.
\newblock Multivariate positive definite functions on spheres.
\newblock 2008.

\bibitem{KN4}
O.R. Musin.
\newblock The kissing number in four dimensions.
\newblock {\em Annals of Mathematics}, 168:1--32, 2008.

\bibitem{LB}
G.~Nebe and N.~J.~A. Sloane.
\newblock Table of the highest kissing numbers presently known.
\newblock
  \url{http://www.math.rwth-aachen.de/~Gabriele.Nebe/LATTICES/kiss.html}, 2011.

\bibitem{KN824}
A.M. Odlyzko and N.J.A. Sloane.
\newblock New bounds on the number of unit spheres that can touch a unit sphere
  in n dimensions.
\newblock {\em J. Combin. Theory Ser. A}, 26:210--214, 1979.

\bibitem{Parrilo}
P.~Parrilo.
\newblock Structured semidefinite programming and algebraic geometry methods in
  robustness and optimization.
\newblock {\em Ph.D. thesis, California Institute of Technology, Pasadena, CA},
  2000.

\bibitem{alphaPVZ}
J.~Pe{\~n}a, J.C. Vera, and L.F. Zuluaga.
\newblock Computing the stability number of a graph via linear and semidefinite
  programming.
\newblock {\em SIAM J. Optim.}, 18(2):87--105, 2007.

\bibitem{Pena15}
J.~Pe{\~n}a, J.C. Vera, and L.F. Zuluaga.
\newblock Completely positive reformulations for polynomial optimization.
\newblock {\em Math. Program.}, 151(2):405--431, 2015.

\bibitem{Pfender}
F.~Pfender.
\newblock Improved {D}elsarte bounds for spherical codes in small dimensions.
\newblock {\em Journal of Combinatorial Theory}, A 114(6):1133--1147, 2007.

\bibitem{FracChromRef}
J.~Povh and F.~Rendl.
\newblock Copositive programming motivated bounds on the stability and the
  chromatic numbers.
\newblock {\em Discrete Optimization}, 6:231--241, 2009.

\bibitem{Reznick}
V.~Powers and B.~Reznick.
\newblock A new bound for p\'{o}lya's theorem with applications to polynomials
  positive on polyhedra.
\newblock {\em Journal of Pure and Applied Algebra}, 164(1-2):221--229, 2001.

\bibitem{Roebers}
L.~Roebers.
\newblock Bounds on the generalized kissing number.
\newblock {\em Master's thesis, Tilburg University}, 2018.

\bibitem{Scherer}
C.~W. Scherer and C.~W.~J. Hol.
\newblock Matrix sum-of-squares relaxations for robust semi-definite programs.
\newblock {\em Math. Program.}, 107(1):189--211, 2006.

\bibitem{Schoenberg}
I.~J. Schoenberg.
\newblock Positive definite functions on spheres.
\newblock {\em Duke Math. J.}, 9(1):96--108, 1942.

\bibitem{KN3}
K.~Sch{\"u}tte and B.L. van~der Waerden.
\newblock Das problem der dreizehn kugeln.
\newblock {\em Mathematische Annalen}, 125:325--334, 1953.

\bibitem{coposKKT}
W.~Xia, J.C. Vera, and L.~F. Zuluaga.
\newblock Globally solving non-convex quadratic programs via linear integer
  programming techniques.
\newblock \url{https://arxiv.org/abs/1511.02423}, 2015.

\bibitem{Xu18}
G.~Xu and S.~Burer.
\newblock A copositive approach for two-stage adjustable robust optimization
  with uncertain right-hand sides.
\newblock {\em Computational Optimization and Applications}, 70(1):33--59,
  2018.

\bibitem{Outer_Y}
E.A. Yildirim.
\newblock On the accuracy of uniform polyhedral approximations of the
  copositive cone.
\newblock {\em Optim. Methods Softw.}, 27:155--173, 2012.

\end{thebibliography}

\end{document}